\documentclass{amsart}
\usepackage{amssymb, amsbsy, amsthm, amsmath, amstext, amsopn, verbatim}
\usepackage[all]{xy}
\usepackage{amsfonts}
\usepackage{amscd}
\usepackage{mathrsfs}
\usepackage{hyperref}
\usepackage{verbatim}
\usepackage{array,multirow,booktabs,longtable}
\RequirePackage{tikz-cd}
\usepackage{tikz}

\newtheorem{thm}{Theorem} [section]
\newtheorem{lemma}[thm]{Lemma}

\newtheorem{lem}[thm]{Lemma}

\newtheorem{corollary}[thm]{Corollary}
\newtheorem{prop}[thm]{Proposition}

\newtheorem*{basic assumption}{Basic Assumption}

\theoremstyle{definition}

\newtheorem*{principal example}{Main Example}

\newtheorem{defn}[thm]{Definition}

\newtheorem{example}[thm]{Example}

\newtheorem{question}[thm]{Question}

\newtheorem{assumption}[thm]{Assumption}
\newtheorem{exmp}[thm]{Example}

\theoremstyle{remark}

\newtheorem{remark}[thm]{Remark}
                         
\newtheorem{claim}[thm]{Claim}

\DeclareMathOperator{\Ind}{Ind} 
\DeclareMathOperator{\Id}{Id} 
\DeclareMathOperator{\Res}{Res}

\begin{document}

\numberwithin{equation}{section}

\newcommand{\hs}{\mbox{\hspace{.4em}}}
\newcommand{\bd}{\begin{displaymath}}
\newcommand{\ed}{\end{displaymath}}
\newcommand{\bcd}{\begin{CD}}
\newcommand{\ecd}{\end{CD}}

\newcommand{\proj}{\operatorname{Proj}}
\newcommand{\bproj}{\underline{\operatorname{Proj}}}
\newcommand{\spec}{\operatorname{Spec}}
\newcommand{\bspec}{\underline{\operatorname{Spec}}}
\newcommand{\pline}{{\mathbf P} ^1}
\newcommand{\pplane}{{\mathbf P}^2}
\newcommand{\coker}{{\operatorname{coker}}}
\newcommand{\ldb}{[[}
\newcommand{\rdb}{]]}

\newcommand{\Sym}{\operatorname{Sym}^{\bullet}}
\newcommand{\Symp}{\operatorname{Sym}}
\newcommand{\Pic}{\operatorname{Pic}}
\newcommand{\AAut}{\operatorname{Aut}}
\newcommand{\PAut}{\operatorname{PAut}}

\newcommand{\too}{\twoheadrightarrow}
\newcommand{\C}{{\mathbb C}}
\newcommand{\cA}{{\mathcal A}}
\newcommand{\cS}{{\mathcal S}}
\newcommand{\cV}{{\mathcal V}}
\newcommand{\cM}{{\mathcal M}}
\newcommand{\bA}{{\mathbf A}}
\newcommand{\cB}{{\mathcal B}}
\newcommand{\cC}{{\mathcal C}}
\newcommand{\cD}{{\mathcal D}}
\newcommand{\D}{{\mathcal D}}
\newcommand{\boldc}{{\mathbf C}}
\newcommand{\cE}{{\mathcal E}}
\newcommand{\cF}{{\mathcal F}}
\newcommand{\cG}{{\mathcal G}}
\newcommand{\G}{{\mathbf G}}

\newcommand{\bH}{{\mathbf H}}
\newcommand{\cH}{{\mathcal H}}
\newcommand{\cI}{{\mathcal I}}
\newcommand{\cJ}{{\mathcal J}}
\newcommand{\cK}{{\mathcal K}}
\newcommand{\cL}{{\mathcal L}}
\newcommand{\baL}{{\overline{\mathcal L}}}
\newcommand{\M}{{\mathcal M}}
\newcommand{\bM}{{\mathbf M}}
\newcommand{\bm}{{\mathbf m}}
\newcommand{\cN}{{\mathcal N}}
\newcommand{\theo}{\mathcal{O}}
\newcommand{\cP}{{\mathcal P}}
\newcommand{\cR}{{\mathcal R}}
\newcommand{\boldp}{{\mathbf P}}
\newcommand{\boldq}{{\mathbf Q}}
\newcommand{\bbL}{{\mathbf L}}
\newcommand{\cQ}{{\mathcal Q}}
\newcommand{\cO}{{\mathcal O}}
\newcommand{\Oo}{{\mathcal O}}
\newcommand{\OX}{{\Oo_X}}
\newcommand{\OY}{{\Oo_Y}}
\newcommand{\dd}{\mathcal{D}}
\newcommand{\Hamp}{\mathbb{H}^{\perp}}
\newcommand{\Ham}{\mathbb{H}}
\newcommand{\cX}{{\mathcal X}}
\newcommand{\cW}{{\mathcal W}}
\newcommand{\boldz}{{\mathbf Z}}
\newcommand{\cZ}{{\mathcal Z}}
\newcommand{\qgr}{\operatorname{qgr}}
\newcommand{\gr}{\operatorname{gr}}
\newcommand{\coh}{\operatorname{coh}}
\newcommand{\End}{\operatorname{End}}
\newcommand{\Hom}{\operatorname{Hom}}
\newcommand{\IndCoh}{\operatorname{IndCoh}}
\newcommand{\sHom}{\mathcal{H}om}
\newcommand{\sEnd}{\mathcal{E}nd}
\newcommand{\uHom}{\underline{\operatorname{Hom}}}
\newcommand{\uHomY}{\uHom_{\OY}}
\newcommand{\uHomX}{\uHom_{\OX}}
\newcommand{\Ext}{\operatorname{Ext}}
\newcommand{\bExt}{\operatorname{\bf{Ext}}}
\newcommand{\Tor}{\operatorname{Tor}}

\newcommand{\inv}{^{-1}}
\newcommand{\airtilde}{\widetilde{\hspace{.5em}}}
\newcommand{\airhat}{\widehat{\hspace{.5em}}}
\newcommand{\nt}{^{\circ}}
\newcommand{\del}{\partial}

\newcommand{\supp}{\operatorname{supp}}
\newcommand{\GK}{\operatorname{GK-dim}}
\newcommand{\W}{W}
\newcommand{\id}{\operatorname{id}}
\newcommand{\res}{\operatorname{res}}
\newcommand{\lrar}{\leadsto}
\newcommand{\im}{\operatorname{Im}}
\newcommand{\HH}{\operatorname{H}}
\newcommand{\Coh}[1]{#1\text{-}{\mathsf{coh}}}
\newcommand{\QCoh}[1]{#1\text{-}{\mathsf{qcoh}}}
\newcommand{\PCoh}[1]{#1\text{-}{\mathsf{procoh}}}
\newcommand{\Good}[1]{#1\text{-}{\mathsf{good}}}
\newcommand{\QGood}[1]{#1\text{-}{\mathsf{Qgood}}}
\newcommand{\Hol}[1]{#1\text{-}{\mathsf{hol}}}
\newcommand{\Reghol}[1]{#1\text{-}{\mathsf{reghol}}}
\newcommand{\Bun}{\operatorname{Bun}}
\newcommand{\Hilb}{\operatorname{Hilb}}
\newcommand{\pa}{\partial}
\newcommand{\F}{\mathcal{F}}
\newcommand{\nthord}{^{(n)}}
\newcommand{\Aut}{\underline{\operatorname{Aut}}}
\newcommand{\Gr}{\operatorname{\bf Gr}}
\newcommand{\Fr}{\operatorname{Fr}}
\newcommand{\GL}{\operatorname{GL}}
\newcommand{\gl}{\mathfrak{gl}}
\newcommand{\SL}{\operatorname{SL}}
\newcommand{\ff}{\footnote}
\newcommand{\ot}{\otimes}
\newcommand{\Wx}{\mathcal W_{\mathfrak X}}
\newcommand{\gh}{\text{gr}_{\hbar}}
\newcommand{\ig}{\iota_g}
\def\Ext{\operatorname {Ext}}
\def\Hom{\operatorname {Hom}}

\def\bbZ{{\mathbb Z}}
\newcommand{\iso}{{\;\stackrel{_\sim}{\to}\;}}

\newcommand{\nc}{\newcommand}
\newcommand{\on}{\operatorname}
\nc{\cont}{\on{cont}}
\nc{\rmod}{\on{mod}}
\nc{\Mtil}{\widetilde{M}}
\nc{\wb}{\overline}
\nc{\wt}{\widetilde}
\nc{\wh}{\widehat}
\nc{\sm}{\setminus}
\nc{\mc}{\mathcal}
\nc{\mbb}{\mathbb}
\nc{\Mbar}{\wb{M}}
\nc{\Nbar}{\wb{N}}
\nc{\Mhat}{\wh{M}}
\nc{\pihat}{\wh{\pi}}
\nc{\opp}{\mathrm{opp}}
\nc{\phitil}{\wt{\phi}}
\nc{\Qbar}{\wb{Q}}
\nc{\DYX}{\D_{Y\leftarrow X}}
\nc{\DXY}{\D_{X\to Y}}
\nc{\dR}{\stackrel{\bbL}{\underset{\D_X}{\ot}}}
\nc{\Winfi}{\cW_{1+\infty}}
\nc{\K}{{\mc K}}
\nc{\unit}{{\bf \on{unit}}}
\nc{\boxt}{\boxtimes}
\nc{\xarr}{\stackrel{\rightarrow}{x}}
\nc{\Cnatbar}{\overline{C}^{\natural}}
\nc{\oJac}{\overline{\on{Jac}}}
\nc{\gm}{{\mathbf G}_m}
\nc{\Loc}{\on{Loc}}
\nc{\Bm}{\operatorname{Bimod}}
\nc{\lie}{{\mathfrak g}}
\nc{\lb}{{\mathfrak b}}
\nc{\lien}{{\mathfrak n}}
\nc{\E}{\mathcal{E}}
\nc{\Cs}{\mathbb{G}_m}
\nc{\hol}{\mathrm{Hol}}
\nc{\can}{\mathrm{can}}
\newcommand{\idot}{{\:\raisebox{2pt}{\text{\circle*{1.5}}}}}

\nc{\Gm}{{\mathbb G}_m}
\nc{\Gabar}{\wb{\G}_a}
\nc{\Gmbar}{\wb{\G}_m}
\nc{\PD}{{\mathbb P}_{\D}}
\nc{\Pbul}{P_{\bullet}}
\nc{\PDl}{{\mathbb P}_{\D(\lambda)}}
\nc{\PLoc}{\mathsf{MLoc}}
\nc{\Tors}{\on{Tors}}
\nc{\PS}{{\mathsf{PS}}}
\nc{\PB}{{\mathsf{MB}}}
\nc{\Pb}{{\underline{\operatorname{MBun}}}}
\nc{\Ht}{\mathsf{H}}
\nc{\bbH}{\mathbb H}
\nc{\gen}{^\circ}
\nc{\Jac}{\operatorname{Jac}}
\nc{\sP}{\mathsf{P}}
\nc{\otc}{^{\otimes c}}
\nc{\Det}{\mathsf{det}}
\nc{\PL}{\on{ML}}

\nc{\ml}{{\mathcal S}}
\nc{\Xc}{X_{\on{con}}}
\nc{\sgood}{\text{strongly good}}
\nc{\Xs}{X_{\on{strcon}}}
\nc{\resol}{\mathfrak{X}}
\nc{\map}{\mathsf{f}}
\nc{\tor}{\mathrm{tor}}
\nc{\base}{Z}
\nc{\bigvar}{\mathsf{W}}
\nc{\alg}{\mathsf{A}}
\nc{\T}{\mathsf{T}}
\nc{\qcoh}{\on{qcoh}}
\renewcommand{\o}{\otimes}
\nc{\mf}{\mathfrak}
\nc{\NN}{\mathsf{N}}
\nc{\h}{\hbar}
\nc{\ms}{\mathscr}
\nc{\good}{\mathrm{good}}
\newcommand{\LMod}[1]{#1\text{-}{\mathsf{Mod}}}
\newcommand{\Lmod}[1]{#1\text{-}{\mathsf{mod}}}
\nc{\mbf}{\mathbf}
\nc{\ad}{\mathrm{ad}}
\nc{\Rees}{\mathsf{Rees}}
\nc{\Supp}{\mathrm{Supp}}
\nc{\Z}{\mathbb{Z}}
\nc{\N}{\mathbb{N}}
\nc{\ann}{\mathrm{ann}}
\nc{\Blf}{B_{\mathrm{l.f.}}}
\nc{\bx}{\mathbf{x}}
\nc{\by}{\mathbf{y}}
\nc{\bz}{\mathbf{z}}
\nc{\bw}{\mathbf{w}}
\nc{\Der}{\mathrm{Der}}
\nc{\CatCs}{\ms{C}}
\renewcommand{\mod}{ \ \mathrm{mod} \ }
\nc{\sA}{\mc{A}} 
\nc{\A}{A} 
\nc{\B}{B} 
\nc{\sW}{\mathscr{W}} 
\nc{\rh}{\mathrm{r.h.}}
\nc{\cs}{\mathbb{C}^*}
\nc{\R}{\mathbb{R}}
\nc{\Lie}{\mathrm{Lie}}
\nc{\Lag}{\Lambda^{c}}
\nc{\ds}{\displaystyle}
\nc{\Qcoh}{\mathsf{Qcoh}}
\nc{\WQcoh}{\mathsf{Qcoh}(\cW)}
\nc{\Indf}{\mathsf{Ind}}
\nc{\DR}{\mathsf{DR}}
\nc{\co}{\operatorname{co}}
\nc{\op}{\operatorname{\opp}}
\nc{\Eq}{\mathrm{Eq}}
\renewcommand{\L}{\mathbb{L}}
\newcommand{\gwyn}[1]{\textcolor{brown}{#1}}

\title[Skeleta and Koszul Duality for Bionic Varieties]{Lagrangian Skeleta and Koszul Duality on Bionic Symplectic Varieties}
\author{Gwyn Bellamy}
\address{School of Mathematics and Statistics\\University of Glasgow\\Glasgow,  Scotland\\G12 8QW.}
\email{gwyn.bellamy@glasgow.ac.uk}
\author{Christopher Dodd}
\address{Department of Mathematics\\University of Illinois at Urbana-Champaign\\Urbana, IL 61801 USA}
\email{csdodd2@illinois.edu}
\author{Kevin McGerty}
\address{Mathematical Institute\\University of Oxford\\Oxford, England, UK}
\email{mcgerty@maths.ox.ac.uk}
\author{Thomas Nevins}
\address{Department of Mathematics\\University of Illinois at Urbana-Champaign\\Urbana, IL 61801 USA}
\email{nevins@illinois.edu}

\begin{abstract}
We consider the category of modules over sheaves of Deformation-Quantization (DQ) algebras on bionic symplectic varieties. These spaces are equipped with both an elliptic $\Gm$-action and a Hamiltonian $\Gm$-action, with finitely many fixed points. On these spaces one can consider geometric category $\mathcal{O}$: the category of (holonomic) modules supported on the Lagrangian attracting set of the Hamiltonian action. We show that there exists a local generator in geometric category $\mathcal{O}$ whose dg endomorphism ring, cohomologically supported on the Lagrangian attracting set, is derived equivalent to the category of all DQ-modules. This is a version of Koszul duality generalizing the equivalence between $\dd$-modules on a smooth variety and dg-modules over the de Rham complex. 
\end{abstract}


\maketitle



\section{Introduction}
The Riemann-Hilbert correspondence provides a fundamental link between the algebra of systems of differential equations and the topology of complex
algebraic varieties.  Its extension to an equivalence between categories of algebraic $\D$-modules on a smooth quasiprojective variety $X$ and dg modules over the de Rham complex $\Omega_X$ in \cite{Kap} represents a deep example of Koszul duality for noncommutative quadratic algebras. 
The present paper extends the duality between $\D$-modules and $\Omega$-modules to a more flexible symplectic context, and deduces a locality result---``sheafification over the Lagrangian skeleton''---that expresses a fundamental Morse-theoretic property of the category of modules over a deformation quantization.  

Let $\resol$ be a smooth, complex algebraic variety with algebraic symplectic form $\omega$.  Such varieties that appear in geometric representation theory typically come equipped with an action of the multiplicative group $\Gm$ of invertible complex numbers that rescales the symplectic form with positive weight: 
$t^*\omega = t^{\ell}\omega$ for some $\ell >0$.  We call such an action {\em elliptic} when the limit $\displaystyle \lim_{t\rightarrow\infty} t\cdot x$ exists for all $x\in \resol$.  Typical examples of symplectic varieties with elliptic $\Gm$-action include cotangent bundles of flag varieties and more general projective spherical varieties, and Nakajima quiver varieties.  

Many symplectic varieties that arise in nature can be equipped not only with an elliptic $\Gm$-action but also a compatible {\em Hamiltonian} $\Gm$-action, i.e. one preserving the symplectic form.  When the two actions have the same, finite, set of fixed points, we call the variety $\resol$ {\em bionic symplectic}: a precise definition appears in Section \ref{properties}.  Standard examples are, again, cotangent bundles of projective spherical varieties, and Hilbert schemes of points on minimal resolutions of cyclic Kleinian singularities (among other quiver varieties).  
Such a variety comes equipped with a canonical Lagrangian ``skeleton'' $\Lag$.  

Suppose now that a symplectic variety $\resol$ comes equipped with a deformation quantization $\cW$ compatible with an elliptic $\Gm$-action.  
Taking $\cW$-modules 
equivariant for the elliptic $\Gm$-action, one arrives at a category useful for geometric representation theory.  When the action is the scaling action on fibers of $\resol = T^*X$, this category reproduces a category of twisted $\D$-modules on $X$; more generally such categories are closely related (depending on the geometry of $\resol$) to representations of Cherednik algebras, finite $W$-algebras, and more.  When the variety is bionic symplectic and the deformation quantization is compatible with both $\Gm$-actions, there is a good geometric analogue of category $\theo$ \cite{HypertoriccatO,BLPWAst,LosevCatOquant,GenCatOWebster} inside the category of $\cW$-modules, consisting of modules supported on the skeleton $\Lag$.  

The present paper establishes two basic features of the category of $\cW$-modules on a bionic symplectic variety.  We build a sheaf of proper dg algebras $\Omega$ that plays the role for $\cW$-modules analogous to that played by  the de Rham complex for $\D_X$-modules, i.e. $\cW$-modules on $\resol = T^*X$. The sheaf $\Omega$ is supported (cohomologically) on a Lagrangian $\Lambda \supset \Lag$ in $\resol$; while for general bionic symplectic varieties we do not have an explicit description of $\Lambda$, in the case where $\resol$ is a symplectic resolution (which is the case in essentially all examples of import to geometric representation theory), we have $\Lambda = \Lag$ (c.f. \ref{thm:symprescohsupport} below). 



\begin{thm}\label{thm:main1}
\mbox{}Assume that locally free resolutions exist for coherent $\cW$-modules. Then the bounded derived category of coherent $\cW$-modules is equivalent to the derived category of coherent dg $\Omega$-modules:
\[
D^b(\Good{\cW}) \iso D(\Lmod{\Omega}).
\]
\end{thm}

As in \cite{BDMN}, we define the category of "quasi-coherent $\cW$-modules'' $\QCoh{\cW}$ to be the ind-category of good $\cW$-modules. We show that Theorem~\ref{thm:main1} extends to the derived category $D(\QCoh{\cW})$, though not in the obvious way. We introduce the "exotic'' derived category $D_{ex}(\LMod{\Omega})$ as a certain quotient of $K(\LMod{\Omega})$, admitting a quotient functor $D_{ex}(\LMod{\Omega}) \to D(\LMod{\Omega})$  to the usual derived category of dg $\Omega$-modules. 

\begin{thm}\label{thm:main3}
    The category $D(\Lmod{\Omega})$ is a full (triangulated) subcategory of $D_{ex}(\LMod{\Omega})$ such that the equivalence of Theorem~\ref{thm:main1} extends to an equivalence
    \[
    D(\QCoh{\cW}) \iso D_{ex}(\LMod{\Omega}).
\]
\end{thm}

In fact, we first establish the equivalence of Theorem~\ref{thm:main3} and then show that this restricts to the equivalence given in Theorem~\ref{thm:main1}.

Since any good $\cW$-module locally admits a resolution by free $\cW$-modules, we derive, as a formal consequence, the following ``sheafification over the Lagrangian skeleton'':
\begin{thm}\label{thm:main2}
The bounded derived category of coherent $\cW$-modules on $\resol$ is the category of global sections of a sheaf of dg categories over $\Lambda$.  
\end{thm}
Beyond \cite{Kap}, the paper draws inspiration from the structure of Fukaya categories of Weinstein manifolds in real symplectic geometry 
as layed out by
Kontsevich \cite{Ko}, explored in important works by several authors (including, among others, Abouzaid, Seidel, Soibelman, Tamarkin, Tsygan, Zaslow)
and given flesh in a categorical context in work of 
Nadler \cite{Nadler}.  Links between Fukaya categories and deformation quantizations form an established theme in the field.  Theorems 
\ref{thm:main1} and, especially, \ref{thm:main2} may be viewed as furthering those connections. 

\subsection{Symplectic resolutions}

In the case where $\resol$ is a symplectic resolution of its affinization $X := \resol^{\mathrm{aff}}$, it was shown in \cite{BLPWAst} that, possibly after shifting by a quantized line bundle, we may assume localization holds. This implies that bounded locally free resolutions exist for coherent $\cW$-modules. Therefore, the assumptions of Theorem~\ref{thm:main1} are satisfied. Moreover, in this situation we have greater control over the cohomological support of $\Omega$. Namely, 

\begin{thm}\label{thm:symprescohsupport}
Assume that $\resol \to X$ is a symplectic resolution. Then the sheaf of dg algebras $\Omega$ is cohomologically supported on the canonical Lagrangian skeleton $\Lag$. 
\end{thm}

The key to Theorem~\ref{thm:symprescohsupport} is the fact that one can choose a local generator $\mc{L}$ (whose dg endomorphism ring is $\Omega$) supported on $\Lag$; see Proposition~\ref{simple extension prop}.  

\subsection{Extension of holonomic complexes}

In order to prove the main theorem we require a number of auxiliary results that are of independent interest. We establish a full recollmenent pattern for holonomic modules on symplectic varieties with elliptic action. Let $C$ be a closed coisotropic cell (this means that $C$ is an attracting set for the  elliptic action) and $U$ its complement. We form the usual diagram
\[
C \stackrel{i}{\hookrightarrow} \resol \stackrel{j}{\hookleftarrow} U.
\]
Let $D^b_{\hol,C}(\cW_{\resol})$ denote the bounded derived category of $\cW$-modules with holonomic cohomology supported on $C$. As noted in the introduction to \cite{BDMN}, we show that:

\begin{thm}
There is a functor $\R i^! \colon D^b_{\hol}(\cW_{\resol}) \to D^b_{\hol,C}(\cW_{\resol})$. Equivalently, $\R j_*$ restricts to a functor $\R j_* \colon D^b_{\hol}(\cW_{U}) \to D^b_{\hol}(\cW_{\resol})$.   
\end{thm}

Since the duality functor $\mathbb{D}$ preserves holonomicity, we can set 
\[
i^* = \mathbb{D} \circ \R i^! \circ \mathbb{D}, \quad j_! = \mathbb{D} \circ \R j_* \circ \mathbb{D},
\]
and deduce the following. 

\begin{corollary}
    There is a full recollment pattern
    \[
    \begin{tikzcd}
    D^b_{\hol,C}(\cW_{\resol}) \ar[rr,"{i_*}"] & & \ar[ll,"{\R i^!}"',bend right=20] \ar[ll,"{i^*}",bend left=15] D^b_{\hol}(\cW_{\resol}) \ar[rr,"{j^*}"] & &  D^b_{\hol}(\cW_{U})  \ar[ll,"{\R j_*}"',bend right=20] \ar[ll,"{j_!}",bend left=15].    
    \end{tikzcd}
    \]
\end{corollary}

If $S$ denotes the coisotropic reduction of $C$ then it is shown in \cite[Theorem~1.7(2)]{BDMN} that the functor of Hamiltonian reduction $\Ham$ identifies $D^b_{\hol,C}(\cW_{\resol})$ with $D^b_{\hol}(\cW_{S})$. 

The proof of these results is given in Section~\ref{sec:preserveholonomic}. 

\subsection{Categorical conventions}\label{sec:categoryconvention}

Let $\mc{A}$ be a sheaf of algebras. If $\ms{M},\ms{N}$ are sheaves of $\mc{A}$-modules then $\sHom_{\mc{A}}(\ms{M},\ms{N})$ is the sheaf of homomorphisms from $\ms{M}$ to $\ms{N}$. $C(\mc{A})$ will denote the (abelian) category of complexes of $\mc{A}$-modules, $K(\mc{A})$ its homotopy category and $D(\mc{A})$ the unbounded derived category of $\mc{A}$-modules. If $\mc{E},\mc{F} \in C(\mc{A})$ then $\underline{\sHom}(\mc{E},\mc{F})$ denotes the complex with terms $\sHom_{C(\mc{A})}(\mc{E},\mc{F}[k])$, and the usual differential. Internal hom in $D(\mc{A})$ is denoted $\R \sHom(\mc{E},\mc{F})$. 

If $\Omega$ is a sheaf of dg-algebras, with $\Omega^i = 0$ for $|i| \gg 0$, then $\LMod{\Omega}$ is the (abelian) category of sheaves of dg $\Omega$-modules and $\Lmod{\Omega}$ denotes the category of coherent dg $\Omega$-modules. Then $K(\Lmod{\Omega})$ is the associated homotopy category and $D(\Lmod{\Omega})$ its derived category. 




\subsection*{Acknowledgments}
The authors are grateful to D. Ben-Zvi, K. Kremnitzer, J. Lebowski, D. Nadler, and D. Treumann for inspiration and helpful conversations.  We would also like to thank the two referees who read the first version of the paper and provided many helpful comments, which undoubtedly improved the exposition in many places.

The first author was partially supported by a Research Project Grant
from the Leverhulme Trust and by the EPSRC grants EP-W013053-1 and EP-R034826-1. The third author was supported by a Royal Society research fellowship.  The fourth author was supported by NSF grants DMS-0757987 and DMS-1159468 and NSA grant H98230-12-1-0216, and by an All Souls Visiting Fellowship.  All four authors were supported by MSRI and by the rising moon over Port Meadow.

\section{Bionic Symplectic Varieties}

In this section we introduce a class of elliptic symplectic varieties which come naturally equipped with a Lagrangian skeleton, and for which we will be able to show the existence of a local generator in the sense of Definition~\ref{defn:localgenerator}.

\subsection{Lagrangian Skeletons}
\label{sec:elliptic}
We begin by recalling from \cite{BDMN} the notion of an elliptic symplectic variety.

\begin{defn}
Let $(\resol,\omega)$ be a smooth symplectic variety equipped with an action $\lambda \colon \mathbb G_m\to \text{Aut}(\resol)$ of the multiplicative group. We say that $\resol$ is \textit{elliptic symplectic} if
\begin{itemize}
\item[$i)$]
The symplectic form $\omega$ has weight $\ell > 0$, \textit{i.e.} $\lambda(t)^*(\omega) = t^\ell \omega$, ($t \in \mathbb G_m(\C)$),
\item[$ii)$]
For each $x \in \resol$ the limit $\lim_{t \to \infty} \lambda(t).x$ exists.
\end{itemize}
\end{defn}

Given an elliptic symplectic variety, the fixed point locus $Y$ of the $\mathbb G_m$-action can be written as a union of connected components $Y = \bigsqcup_{i \in I} Y_i$ and this induces a partition $\resol = \bigsqcup_{i \in I} C_i$ of $X$, where $C_i = \{x \in X: \lim_{t \to\infty} \lambda(t).x \in Y_i\}$. We call the $C_i$ the \textit{coisotropic cells} of the partition. As in Lemma 2.4 of \cite{BDMN}, the collection of coisotropic cells are partially ordered by the relation defined by $C_i \geq C_j$ if $\overline{C}_j \cap C_i \neq \emptyset$. Moreover, for any choice of refinement of this order (to a partial or total order), and any index $i$, the subsets $K_{\geq i} = \bigsqcup_{j \geq i} C_j$ and $K_{>i} = \bigsqcup_{j>i} C_j$ are closed coisotropic subsets of $\resol$. 

The following is one of the main results of \cite{BDMN}.

\begin{thm}\label{thm:coistropicreduction}
Let $(\resol,\omega)$ be an elliptic symplectic variety with associated partition $\resol = \bigsqcup C_i$. Then each $C_i$ is a smooth locally-closed $\mathbb G_m$-invariant coisotropic subvariety of $X$ with a $\mathbb G_m$-equivariant coisotropic reduction $(S_i,\omega_i)$. That is, there is an equivariant morphism $\pi_i \colon C_i \to S_i$ where $S_i$ is an elliptic symplectic variety, and $\pi_i^*(\omega_i) = \omega_{|C_i}$.
\end{thm}

One of the goals of the present paper is to understand to what extent the category of $\mathbb G_m$-equivariant modules for a quantization $\mathcal W_\resol$ of $\resol$ can be described in terms of objects supported on Lagrangians in $\resol$. The basic notion we use for this is that of a Lagrangian skeleton.

\begin{defn}
Let $(\resol,\omega)$ be an elliptic symplectic variety, and $\Lambda$ a closed (possibly singular) Lagrangian subvariety of $\resol$. We say that $\Lambda$ is a \textit{Lagrangian skeleton} for $\resol$ if
\begin{itemize}
\item[$i)$]
$\Lambda$ is $\mathbb G_m$-invariant.
\item[$ii)$]
For each $i \in I$, $\Lambda_i = \Lambda\cap C_i$ is a nonempty smooth closed subvariety of $C_i$.
\end{itemize}
\end{defn}

We will refer to $\Lambda_i$ as a \textit{Lagrangian cell}. Using the partial ordering of the coisotropic cells (and any refinement of it) we obtain closed Lagrangian subvarieties $\Lambda_{\geq i} = \bigsqcup_{j \geq i} \Lambda_j$ and $\Lambda_{>i} = \bigsqcup_{j>i} \Lambda_j$. Note that under our assumptions the closures of the $\Lambda_i$ are the irreducible components of $\Lambda$. 
The next lemma shows that a Lagrangian skeleton is compatible with the collection of symplectic reductions $\pi_i\colon C_i \to S_i$.

\begin{lemma}
If $(\resol,\omega)$ is elliptic symplectic, and $\Lambda$ is a closed $\mathbb G_m$-invariant Lagrangian subvariety such that $\Lambda \cap C_i = \Lambda_i$ is smooth, then $\pi_i(\Lambda_i)$ is a smooth Lagrangian in $S_i$ and $\Lambda_i = \pi_i^{-1}(\pi_i({\Lambda}_i))$.
\end{lemma}

\begin{proof}
We must check that $\Lambda_i = \pi_i^{-1}(\pi_i(\Lambda_i))$, and that $\pi_i(\Lambda_i)$ is Lagrangian in $S_i$. This can be checked locally on $S_i$, and indeed since it is shown in \cite[\S 2]{BDMN} that $\pi_i$ factors the projection map $p_i\colon C_i \to Y_i$, it is enough to check this locally on $Y_i$. We thus fix $y \in Y_i$. Now $p_i$ is an affine bundle over $Y_i$, and by \cite{BBFix} we may equivariantly trivialize it near $y$, so that we may assume $C_i = Y_i\times V$ where $V$ is a linear $\mathbb G_m$-representation with strictly negative weight. If $U<V$ denotes the subspace of vectors with weight $k$ where $-\ell \leq k <0$ and $W$ denotes the subspace of $V$ consisting of vectors with weights strictly less than $-\ell$, then it is shown in Lemma 2.5 of \cite{BDMN} that in this trivialization we have $S_i = Y_i\times U$, where $\pi_i$ is the projection along $W$. 

Now consider $\Lambda_i$: since $\Lambda_i$ is closed in $C_i$ we may assume $y \in \Lambda_i$. Then we may identify $T_yC_i = T_yY_i\times U\times W$, and since $\Lambda_i$ is Lagrangian we must have $W\subseteq T_y\Lambda_i$. But then it follows that (locally) $\pi_i(\Lambda_i)$ is $(Y\cap \Lambda_i)\times (U\cap \Lambda_i)$ which is Lagrangian in $S_i$, and $\Lambda_i = \pi_i^{-1}(\pi_i(\Lambda_i))$.
\end{proof}

\subsection{Better...Stronger...Faster}\label{properties}
Let $\resol$ be a smooth, connected, complex algebraic symplectic variety with symplectic form $\omega$.  We say that $\resol$ is a {\em conic (or conical) symplectic variety} if $\resol$ comes equipped with a $\Gm$-action satisfying:
\begin{enumerate}
\item[(C1)] the $\Gm$-action is elliptic. 
\end{enumerate}
Now, suppose that $\resol$ comes equipped with a
 $\mathsf{T} = \Gm^2$ action, such that:
\begin{enumerate}
\item[(L1)] There is a choice of $\Gm\rightarrow \mathsf{T}$ that satisfies (C1) above.  
\item[(L2)] There is a different choice of $\mathsf{H} \cong \Gm\rightarrow \mathsf{T}$ that is Hamiltonian: i.e., it preserves the symplectic form.
\item[(L3)] Both actions have the same, finite, set $\cS$ of fixed points.
\end{enumerate}
\begin{defn}
We will call a variety with $\mathsf{T} = \Gm^2$-action satisfying (C1), (L1), (L2), (L3) above a {\em bionic symplectic variety}.
\end{defn}

\begin{remark}
When one has a bionic symplectic variety, there is a bit of freedom in the choice of the torus $\mathsf{T}$. More precisely, this definition really only requires the action of the product $\mathbb{G}_m \times \mathbb{G}_m$ of the tori giving the elliptic and Hamiltonian actions. This torus will be isogenous to $\mathsf{T}$, but may not be equal to it; for our purposes replacing $\mathsf{T}$ by this torus would be completely harmless. However, the natural torus actions occurring in the examples below often use a $\mathsf{T}$ which is not simply the product of the tori giving the elliptic and Hamiltonian actions, so it seemed sensible to allow this extra flexibility in the definition.
\end{remark}

We call the action in (L1) the {\em elliptic $\Gm$-action}, and the action in (L2) the {\em Hamiltonian $\Gm$-action}. Equipping $\resol$ with only the elliptic action it clearly becomes an elliptic symplectic variety, and hence has a decomposition $X = \bigsqcup_{i \in I} C_i$ into coisotropic pieces. However, since the elliptic action has isolated fixed points, each $C_i$ is in this case an affine space $\mathbb A^{d_i}$, and as is shown \cite[\S 2.6]{BDMN}, in this case the symplectic reductions $S_i$ are symplectic vector spaces equipped with an elliptic linear $\mathbb G_m$-action. Indeed, since the constructions of \cite{BDMN} can be carried out equivariantly, the $C_i$ are in fact $\mathsf T$-invariant and the $S_i$ are $\mathsf T$-representations.



\subsection{The Lagrangian $\Lag$}\label{sec:Lag}
Suppose that $\resol$ is bionic symplectic. Let $\mathsf{H} = \mathsf{H}(s) \colon \Gm\rightarrow \on{Aut}(\resol)$ denote the Hamiltonian $\Gm$-action.  We define the {\em canonical Lagrangian skeleton} $\Lag$ of $\resol$ to be:
\begin{equation}\label{eq:skeleton}
\Lag = \Big\{ x\in\resol \,\Big|\, \lim_{s\rightarrow\infty} \mathsf{H}(s)\cdot x \,\,\text{exists in $\resol$}\Big\}.
\end{equation}
\begin{lemma}
The subset $\Lag$ is a Lagrangian subset of $\resol$.
\end{lemma}
For each $i$, let $\Lag_i = C_i\cap\Lag$, where $C_i$ is a coisotropic cell as in Section \ref{sec:elliptic}.  
\begin{prop}
\mbox{}
\begin{enumerate}
\item Under the isomorphism $C_i \cong \mathbb{A}^{d_i}$ given above, each $\Lag_i$ is a vector subspace of the $\Gm$-representation $C_i$.  
\item Moreover, the projection $\pi_i(\Lag_i)\subset S_i$ is a 
$\mathsf{T}$-stable Lagrangian subspace of the symplectic vector space $S_i$, and $\Lag_i = \pi\inv\big(\pi(\Lag_i)\big)$.
\end{enumerate}
\end{prop}  

 
The next corollary is immediate:

\begin{corollary}
\mbox{}
The variety $\Lag$ is a Lagrangian skeleton of $\resol$.
\end{corollary} 

\begin{exmp} \label{exmp:Kleinian}
An illustrative example is provided by the minimal resolution of the Kleinian singularity of type $A$; this example is discussed further in \ref{sec:Hilbcyclic} and \ref{exmp:typeAholo} below. 

Let $Y_n$ denote the Kleinian singularity of type $A_n$ and let $X_n$ be its minimal resolution, a smooth symplectic surface. (The symplectic form is induced from volume form on $\C^2$). There is a natural two-torus $\mathsf T$ acting on $X_n$, and it has $n$ fixed points $\{p_1,p_2,\ldots,p_n\}$ all lying in the exceptional fibre of the resolution: indeed if the exceptional locus is $D_1\cup D_2\ldots \cup D_{n-1}$ where each $D_i$ is a projective line then $D_i \cap D_{i+1}=\{p_{i+1}\}$ for $1\leq i \leq n-2$ and the other fixed points $p_1$ and $p_n$ lie on $D_1\backslash\{p_1\}$ and $D_{n-1}\backslash\{p_{n-1}\}$ respectively. 

Picking a suitable generic one-parameter subgroup, one obtains an elliptic $\mathbb G_m$-action, which has $n$ coisotropic affine cells: the fixed point $p_n$ attracts a two-dimensional affine space $C_n=\mathbb A^2$, while each of the others has attracting locus isomorphic to $\mathbb A^1$: for $p_i$ where $i>1$ this is the affine line $C_i= D_i \backslash\{p_{i-1}\}$, while for $p_0$ the affine line $C_0$ lies outside the exceptional locus, and is tangent to $C_1$ at $p_0$. 

The Hamiltonian $\mathbb G_m$-action $\mathsf{H} \colon \mathbb G_m \to \mathsf T$ is induced from a Hamiltonian action on $\C^2$, and the locus of points $\Lag$ which have a limit under it is the union of the exceptional locus and an affine line tangent to $C_{n-1}$ at $p_n$. On the open cell, one of course already has a symplectic variety, so it is its own reduction, while on each of the other cells, the reduction is just a point. 
\end{exmp}

We list below standard examples of bionic symplectic varieties. 

\subsection{Cotangent Bundles of Spherical Varieties}

Let $G$ be a connected reductive group and $B$ a Borel subgroup. 

\begin{lemma}
If $X$ is a complete $B$-spherical variety, then $T^*X$ has a $\T$-action making it bionic symplectic.
\end{lemma}

\begin{proof}
We may write $X$ as the disjoint union $X = \bigsqcup_{s\in \mathcal S}\mathcal O_s$ of its $B$-orbits, where $|\mathcal S|<\infty$. To show that $T^*X$ has the structure of a bionic symplectic variety, it suffices to exhibit a $1$-parameter subgroup $\gamma\colon \mathbb G_m \to B$ such that $\gamma(\mathbb G_m)$ has finitely many fixed points on $\mathcal O_s$ for each $s \in \mathcal S$. Indeed assuming such a $\gamma$ exists, it induces a Hamiltonian action on $T^*X$, and by modifying this action by a sufficiently high power of the scaling action on the fibres of $T^*X$ we may ensure that the resulting action on $T^*X$ is fibrewise convergent.

Fix a maximal torus $T$ of $B$ and, for each $s \in \mathcal S$, let $K_s$ be the stabilizer of a point in $\mathcal O_s$ chosen so that $T_s = K_s\cap T$ is a maximal torus of $K_s$ (this is possible since the maximal tori in $B$ are all conjugate). If $\gamma\colon \mathbb G_m\to T$ then 
\[
bK_s \in (B/K_s)^{\gamma(\mathbb G_m)} \iff \forall t\in \mathbb G_m,  \gamma(t).bK_s = bK_s, \iff b^{-1}\gamma b \colon \mathbb G_m \to K_s.
\]
Let $r\colon B\to B/U$ be  the quotient of $B$ by its unipotent radical, and let $\mathcal S^{\text{max}} = \{s \in \mathcal S: \text{rank}(K_s)= \text{rank}(B)\}$. If $s \notin \mathcal S^{\text{max}}$, then $r(K_s)$ is a proper subtorus of $B/U$ and, since $B/U$ is abelian, for any $b \in B$ and $\gamma\colon \mathbb G_m \to T$ generic,
\[
r\circ (b^{-1}\gamma b)(\mathbb G_m)\cap r(K_s)=r\circ\gamma(\mathbb G_m)\cap r(K_s) =\{1\},
\]
and hence any such $\gamma$ will have no fixed points on $\mathcal O_s$. On the other hand, if $s \in \mathcal S^{\text{max}}$, then $B.eK_s= U.T.eK_s = U.eK_s$, so that $\mathcal O_s$ is a $U$-orbit, and hence is an affine space. But $T$ acts with positive weights on $U$, and hence a generic $1$-parameter subgroup $\gamma$ contracts $U$ to $e$ and hence $B/K_s$ to $eK_s$. Thus for $\gamma\colon\mathbb G_m \to T$  suitably generic we see that $\gamma$ will have exactly $|\mathcal S^{\text{max}}|<\infty$ fixed points on $X$. 
\end{proof}

\begin{corollary}
If $X$ is either a partial flag manifold or a toric variety then $T^*X$ is bionic symplectic.
\end{corollary}

\subsection{Hilbert Schemes of Points on Minimal Resolutions of Type $A$ Kleinian Singularities}\label{sec:Hilbcyclic}
Another example of a bionic symplectic variety is the $\mathsf{T}$-action on $(\C^2)^{[n]}$ induced from the action on $\C^2$. More generally, the action on the Hilbert scheme of points on the minimal resolution of the type $A$ Kleinian singularity is bionic. It is explained in \cite{Gordon2} how to identify the Hamiltonian $\mathbb G_m$-action and show that it has finitely many fixed points. 

The algebra of invariant (with respect to the elliptic $\mathbb G_m$-action) global sections of a DQ-algebra on this Hilbert scheme is isomorphic to the spherical subalgebra of a rational Cherednik algebra associated to the wreath product group $\mathfrak{S}_n \wr \mu_{\ell}$; see \cite{Gordon06} 
for a proof of this. When localization holds, geometric category $\mathcal{O}$ is equivalent to the usual (spherical) category $\mathcal{O}$ introduced in \cite{GGOR}. 


\subsection{Hypertoric varieties and Nakajima quiver varieties}

Hypertoric varieties associated to unimodular matrices $A$ and generic stability parameters \cite{HS} provide a rich source of symplectic resolutions of conic symplectic singularities. In many cases, one can choose a Hamiltonian torus action such that the resulting $2$-torus $\mathsf{T}^2$ has finitely many fixed points. See \cite{HypertoriccatO} for details. 

Similarly, Nakajima quiver varieties often give rise to bionic symplectic varieties. Indeed, the example of Section~\ref{sec:Hilbcyclic} can be realised as a Nakajima quiver variety associated to the cyclic quiver. Another large class of examples come from framed finite ADE type quivers. Using standard notation, if $\mathbf{v}$ is the dimension vector and $\mathbf{w}$ the framing vector then we require each weight $\mathbf{w}_i \varpi_i$ to be a minisucle weight for the corresponding simple Lie algebra. In this case, the smooth quiver variety $M(\mathbf{v},\mathbf{w})$ is bionic; see \cite[Theorem~6.2]{Sympdualitysurvey}.

\subsection{Slodowy Slices}

If $\mathfrak{g}$ is a reductive Lie algebra and $e \in \mathfrak{g}$ is a nilpotent element, Slodowy \cite{Sl} has defined a natural transverse slice $S_e$ at $e$ through its orbit $\mathcal O_e$. Let $\mathcal S_e = S_e \cap \mathcal N$ be the intersection of the slice with the nilpotent cone of $\mathfrak{g}$. Pulling back the Springer resolution to this slice one obtains 
a resolution of singularities of $\pi_e \colon \widetilde{\mathcal S}_e \to \mathcal S_e$, which is compatible with the natural symplectic forms on the smooth locus of $\mathcal S_e$ and $\widetilde{\mathcal S}_e$ so that $\pi_e$ is a symplectic resolution.

There is a natural $\mathbb G_m$-action on $\mathcal S_e$ and $\widetilde{\mathcal S}_e$ (for which $\pi_e$ is equivariant) which contracts $\mathcal S_e$ to the element $e$: Indeed if $\mathfrak{s}$ is an $\mathfrak{sl}_2$-triple containing $e$, then there is a one-parameter subgroup $\lambda$ of $G$ parametrizing the maximal torus of the associated $\text{SL}_2$ in $G$. This acts with weight $2$ on $e$, hence if we define $\rho(t)(x) = t^{-2}\lambda(t)(x)$ ($x \in \mathfrak{g})$, then $\rho$ fixes $e$. A tangent space computation shows that it contracts the slice $S_e$ to $e$, and clearly it lifts to an action on $\widetilde{\mathcal S}_e$. Since the action is compatible with the map $\pi_e$ and $\pi^{-1}(e)$ is proper, every $x\in \widetilde{\mathcal S}_e$ has a limit under the action of $\rho(\mathbb G_m)$. Moreover, the action of $\lambda$ on $T^*(\mathcal B)$ is Hamiltonian, and so the symplectic form $\omega$ on $\widetilde{\mathcal S}_e$ is only altered by scaling of the fibre direction. It follows that $\rho(\mathbb G_m)$ has weight $2$ on $\omega$, so that $\rho$ is an elliptic action.
There is a natural $\mathbb G_m$-action on $\mathcal S_e$ and $\widetilde{\mathcal S}_e$ (for which $\pi_e$ is equivariant) which contracts $\mathcal S_e$ to the element $e$: Indeed if $\mathfrak{s}$ is an $\mathfrak{sl}_2$-triple containing $e$, then there is a one-parameter subgroup $\lambda$ of $G$ parametrizing the maximal torus of the associated $\text{SL}_2$ in $G$. This acts with weight $2$ on $e$, hence if we define $\rho(t)(x) = t^{-2}\lambda(t)(x)$ ($x \in \mathfrak{g})$, then $\rho$ fixes $e$. A tangent space computation shows that it contracts the slice $S_e$ to $e$, and clearly it lifts to an action on $\widetilde{\mathcal S}_e$. Since the action is compatible with the map $\pi_e$ and $\pi^{-1}(e)$ is proper, every $x\in \widetilde{\mathcal S}_e$ has a limit under the action of $\rho(\mathbb G_m)$. Moreover, the action of $\lambda$ on $T^*(\mathcal B)$ is Hamiltonian, and so the symplectic form $\omega$ on $\widetilde{\mathcal S}_e$ is only altered by scaling of the fibre direction. It follows that $\rho(\mathbb G_m)$ has weight $2$ on $\omega$, so that $\rho$ is an elliptic action.

Not all $\widetilde{\mathcal S}_e$ appear to be bionic symplectic varieties however, but a large class of them are, as the following shows: Recall that we say a nilpotent $e \in \mathfrak{g}$ is of \textit{standard Levi type} if there is a Levi subalgebra $\mathfrak{g}_0$ containing $e$ in which $e$ is regular. 

\begin{lemma}
Suppose that $e$ is of standard Levi type. Then $\widetilde{\mathcal S}_e$ is a bionic symplectic variety.
\end{lemma}
\begin{proof}
By the previous paragraph, it suffices to construct a Hamiltonian $\mathbb G_m$ action $\rho\colon \mathbb G_m \to \text{Aut}(\widetilde{\mathcal S}_e)$ which commutes with the Slodowy $\mathbb G_m$-action and has finitely many fixed points. 

Let $\mathfrak{g}_0$ be a Levi subalgebra in which $e$ is regular. Choose a $\mathbb G_m$ action $\lambda\colon \mathbb G_m \to G$ (where $G$ is the adjoint group associated to $\mathfrak{g}$) such that $\mathfrak{g}_0$ is precisely the fixed point locus for the action of $\lambda(\mathbb G_m)$ on $\mathfrak{g}$. Since we may pick an $\mathfrak{sl}_2$-triple $\mathfrak{s}$ for $e$ in $\mathfrak{g}_0$, it is clear that the action $\lambda$ commutes with the elliptic action $\rho$ associated to $\mathfrak{s}$ by the preceding paragraph. 

Now since $e \in \mathfrak{g}_0$ is regular, there is a unique Borel subalgebra $\mathfrak{b}_0 \subseteq \mathfrak{g}_0$ containing $e$. But if $\mathcal B$ and $\mathcal B_0$ denote the flag varieties of $\mathfrak{g}$ and $\mathfrak{g}_0$ respectively, then $\mathcal B^{\lambda(\mathbb G_m)}$ is isomorphic to $|W/W_0|$ copies of $\mathcal{B}_0$, where on each component the map $\mathfrak{b}\mapsto \mathfrak{b}\cap \mathfrak{g}_0$ yields the isomorphism.  Moreover, if $x \in \mathfrak{g}_0\cap \mathcal S_e$, then since $e$ is regular in $\mathfrak{g}_0$, the orbit of $x$ under $G_0$ (the Levi subgroup associated to $\mathfrak{g}_0$) lies in the closure of the orbit of $e$, and hence the closure of the $G$-orbit of $e$ intersects (and hence contains) the orbit of $x$. Since $S_e$ is a transverse slice through the $G$-orbit of $e$, this is possible if and only if $x=e$.

Now suppose that $(x,\mathfrak b) \in \widetilde{\mathcal S}_e$ is fixed by $\lambda(\mathbb G_m)$, then $x \in \mathfrak{g}_0\cap \mathcal S_e = \{e\}$ and $\mathfrak{b}\cap \mathfrak{g}_0 = \mathfrak{b}_0$, hence there are exactly $|W/W_0|$ fixed points for the action of $\lambda(\mathbb G_m)$. It follows that $\widetilde{\mathcal S}_e$ is bionic as required.
\end{proof}

\begin{remark}
The condition that $e$ be a nilpotent element of regular Levi type may also be necessary: if the one-parameter subgroup $\lambda$ arises from a one-parameter subgroup of $T\times \mathbb G_m$, where $T$ is a maximal torus of $G$ and $\mathbb G_m$ acts by homotheties on $\mathfrak{g}$, then  the action will be Hamiltonian if and only if the one-parameter subgroup has image in $T$. But then the condition that $\lambda(\mathbb G_m)$ has finitely many fixed points ensures that $\pi_e^{-1}(e)$ is a finite set, which implies that $e$ is regular in the Levi subalgebra given by $\lambda$. Thus if $\text{Aut}(\widetilde{\mathcal S}_e)$ has maximal torus $T\times \mathbb G_m$ (that is, if there are no other one-parameter subgroups up to conjugacy) we obtain the converse to the lemma. 
\end{remark}

The algebra of invariant (with respect to the elliptic $\mathbb G_m$-action) global sections of a DQ-algebra on the resolution of a Slodowy slice is a finite $W$-algebra. 




\section{$\cW$-Modules Supported on $\Lag$}
\subsection{Background on $\cW$-Algebras and Quasicoherent $\cW$-Modules}
Throughout, we follow the exposition and conventions of \cite{BDMN}.  

A sheaf of $\C[\![\h]\!]$-algebras $\sA$ on $\resol$ is said to be a \textit{deformation-quantization} algebra if it is $\h$-flat and $\h$-adically complete, and comes with an isomorphism of Poisson algebras $\sA / \hbar \sA \cong \mc{O}_X$. If $\sA$ is a deformation-quantization algebra on $\resol$, then  $\cW := \C(\!(\h)\!) \o_{\C[\![\h]\!]} \sA$ is a sheaf of $\C(\!(\h)\!)$-algebras on $\resol$ called the {\em $\cW$-algebra associated to $\sA$}. If $M$ is any $\C[\hbar]$-module (or sheaf) then $M_{\hbar}$ will denote $\C[\hbar^{\pm 1}] \o_{\C[\hbar]} M = M[{\hbar}^{-1}]$. In particular, $\cW = \mc{A}_{\hbar}$. 

We assume throughout that our algebra $\cW$ carries an action of $\mathsf{T}$. Since the symplectic form $\omega$ has weight $\ell > 0$ with respect to the elliptic $\Gm$-action and is invariant with respect to $\mathsf{H}$, the element $\hbar$ has weight $\ell$ with respect to $\Gm$ and is invariant with respect to $\mathsf{H}$. Just as in \cite[Section~2.3]{KR}, from this point on we replace $\cW$ by $\C(\!( \h^{1/\ell} ) \!) \o_{\C(\!(\h)\!)} \cW$, $\sA$ by $\C[\![ \h^{1/\ell} ] \!] \o_{\C[\![\h]\!]} \sA$ and $\hbar$ by $\hbar^{1/\ell}$. Therefore, $\hbar$ now has \textbf{weight one} with respect to $\Gm$. After making this replacement, the Poisson bracket on $\sA / \hbar \sA$ is given by
\[
\{ \overline{a}, \overline{b} \} = \frac{1}{\hbar^{\ell}} [a,b] \mod \hbar \sA.
\]
because the bracket $[a,b]$ is now a multiple of $\hbar^{\ell}$ for all $a,b \in \sA$. 

\begin{assumption}
$\Gm$-equivariant $\cW$-modules are always defined with respect to the elliptic $\Gm$-action; since \textbf{all modules} we consider are $\Gm$-equivariant, we suppress explicit mention of it throughout the paper.    
\end{assumption}

We will occasionally refer explicitly to $\mathsf{T}$-equivariant modules, however: this is an additional structure beyond what we uniformly require of modules.

Coherent modules over $\sA$ are defined analogously to coherent sheaves in commutative algebraic geometry except that they must be $\Gm$-equivariant.  A {\em lattice} for a $\cW$-module $\ms{M}$ is a coherent $\Gm$-equivariant $\sA$-submodule $\ms{M}(0)$ such that $\cW\cdot \ms{M}(0) = \ms{M}(0)_{\hbar} = \ms{M}$. A $\cW$-module $\ms{M}$  is {\em good} if it admits a lattice.  The (abelian) category of good $\cW$-modules is denoted $\Good{\cW}$. The ind-category of good $\cW$-modules is written $\WQcoh$ and called, somewhat abusively, the category of {\em quasicoherent $\cW$-modules}.  Its basic properties are described in \cite{BDMN}.

Given an open immersion $j:U\rightarrow \resol$ from a $\mathbb{G}_m$-invariant open set $U$, there is an exact, full, continuous, and essentially surjective restriction functor 
\[
j^*: \Good{\cW} \rightarrow \Good{\cW_U}.
\]
It induces a continuous functor $j^*:\WQcoh\rightarrow \Qcoh(\cW_U)$.  The latter functor admits a continuous right adjoint $j_*: \Qcoh(\cW_U)\rightarrow\WQcoh$, which does not usually preserve coherence.   Moreover, $j^*j_* \ms{M} = \ms{M}$ for all $\ms{M}\in \Qcoh(\cW_U)$; $j^*$ is a quotient functor and $j_*$ is a section. See Section~5.1 of \cite{BDMN} for details.  

Since the category $\Good{\cW}$ is essentially small, $\QCoh{\cW}$ is a Grothendieck category \cite[Theorem~8.6.5(vi)]{KS}. In particular, it has enough injectives \cite[Theorem~9.6.2]{KS} and small inductive limits. By \cite[Corollary~5.9]{BDMN}:

\begin{thm}
    The (unbounded) derived category $D(\QCoh{\cW})$ admits small inductive limits, is compactly generated and 
    \[
    D(\QCoh{\cW})^c = D^b(\Good{\cW}) = \mathsf{perf} (\cW). 
    \]
\end{thm}

Here $\mathsf{perf} (\cW)$ denotes the full, triangulated, subcategory of $D(\QCoh{\cW})$ consisting of all perfect objects; that is, objects locally isomorphic to a bounded complex of finitely generated projective modules. In fact, just as in the commutative case, projective $\cW$-modules are locally free so one can alternatively assume that locally the objects in $\mathsf{perf} (\cW)$ are bounded complexes of free $\cW$-modules of finite rank.   

Let $C \subset \mf{X}$ be a closed coisotropic cell, as in Section~\ref{sec:elliptic}, and $U$ its complement. We form the usual diagram
\begin{equation}\label{eq:poencloseddiagram}
   \begin{tikzcd}
C \ar[r,hook,"i"] & \mf{X} & \ar[l,hook',"j"'] U. 
\end{tikzcd} 
\end{equation}
As in \cite[Proposition~5.1(1)]{BDMN}, for $\ms{M}$ a good $\cW$-module we define $\Gamma_C(\ms{M})$ to be the largest submodule of $\ms{M}$ whose support is contained in $C$. The main result of  \cite{BDMN} is an equivalence $\Ham \colon \Good{\cW}_C \iso \Good{\cW_S}$ between good $\cW$-modules supported on $C$ and good $\cW_S$-modules, where $S$ is the coisotropic reduction of $C$. This defines a functor $i^! = \Ham \circ \Gamma_C \colon \Good{\cW} \to \Good{\cW_S}$ and extends to an exact functor $\R i^! \colon D(\QCoh{\cW_{\resol}}) \to D(\QCoh{\cW_S})$. The following is implicit in \cite{BDMN}.

\begin{prop}
    There exist adjoint pairs $(j^*,\R j_*)$ and $(i_*,\R i^!)$
    \[
\begin{tikzcd}
    D(\QCoh{\cW_S}) \ar[r,bend right=20,"i_*"'] & D(\QCoh{\cW_{\resol}}) \ar[l,bend right=20,"\R i^!"']  \ar[r,bend right=20,"j^*"'] & D(\QCoh{\cW_U}) \ar[l,bend right=20,"\R j_*"'] 
\end{tikzcd}
\]
forming a triangle 
\begin{equation}\label{eq:openclosedtriangle}
    i_* \circ \R i^! \to \Id \to \R j_* j^* \stackrel{[1]}{\longrightarrow}.
\end{equation}
\end{prop}

\begin{proof}
The fact that the pair $(j^*,\R j_*)$ exists is \cite[Corollary~5.9]{BDMN}, and that the pair $(i_*,\R i^!)$ exists is Proposition~5.1 and Lemma~5.3 of \cite{BDMN}. Finally, it is shown in (5.1) of \cite{BDMN} that $j^*$ identifies $D(\QCoh{\cW_U})$ with $D(\QCoh{\cW}) / D_C(\QCoh{\cW})$. Therefore, the other claim is standard, see for instance \cite[Ex. 10.15]{KS}.
\end{proof}

\subsection{Holonomic Objects and Objects Supported on $\Lag$}

By Gabber's theorem, the support of any good $\cW$-module has dimension at least $(1/2)\dim \resol$. A good $\cW$-module is said to be {\em holonomic} if the dimension of its support is exactly $(1/2)\dim \resol$. The category of holonomic $\cW$-modules is denoted $\Hol{\cW}$. The theory of characteristic cycles implies the following. 

\begin{lem}
Let $\ms{M}$ be a holonomic $\cW$-module. Then $\ms{M}$ has finite length. 
\end{lem}







This paper is concerned mainly with those $\mathsf{T}$-equivariant holonomic objects which are supported on $\Lag$.  Consider an open subset $U = \resol\smallsetminus K_{>i}$ for a closed union $K_{>i}$ of coisotropic cells.  Then $C_i$ is closed in $U$ and $\Lag_i$ is closed in $C_i$.  Recall that $\overline{\Lag}_i = \pi_i(\Lag_i)$ where $\pi_i:C_i\rightarrow S_i$ is the coisotropic reduction. In \cite{BDMN} it is shown that if $\mathcal W_\resol$ is a $\mathbb G_m$-equivariant quantization of $\resol$ then each $S_i$ has a natural $\mathbb G_m$-equivariant quantization $\mathcal W_{S_i}$ and there is a reduction functor $\mathbb H_i$ giving an equivalence 
\[
\Ham_i \colon \Good{\cW_U}_{C_i} \iso \Good{\cW_{S_i}}
\]
between the category of good modules supported on $C_i$ and the category of good modules for $\mathcal W_{S_i}$. The analogous statement holds in the dg context as well. 

\begin{prop}\label{prop:uniquesimple}
\mbox{}
\begin{enumerate}
\item
There is a simple $\mathsf{T}$-equivariant holonomic object $\delta_{\Lag_i}$ of $\Good{\cW_U}$ supported on $\Lag_i$. 
\item  Such an object $\delta_{\Lag_i}$ is unique up to twisting by a character of the Hamiltonian $\Gm$-action.  
\item every $\mathsf{T}$-equivariant holonomic object of $\Good{\cW_U}$ supported on $\Lag_i$ is isomorphic to a finite direct sum of twists (for the Hamiltonian $\Gm$-action) of copies of $\delta_{\Lag_i}$.
\end{enumerate}
\end{prop}

\begin{proof}
Let $\cW_{S_i}$ denote the $\Gm$-equivariant $\cW$-algebra on $S_i$. 
By Theorem 1.5 of \cite{BDMN}, checking that each statement of the proposition holds on $\Lag_i \subset U$ is equivalent to checking the corresponding statement on $\overline{\Lag}_i\subset S_i$.  By \cite[Corollary 1.6]{BDMN}, the category of good $\cW_{S_i}$-modules is equivalent to the category of coherent $\D({\mathbb A}^{d_i})$-modules where $2d_i = \dim(S_i)$.  The additional equivariance with respect to the Hamiltonian action on $\resol$ translates to monodromicity of $\D({\mathbb A}^{d_i})$-modules.  Hence the claim is that there is a unique simple monodromic $\D({\mathbb A}^{d_i})$-module supported (microlocally) on the ``attracting locus'' in $T^*V$ where $V$ is a $\Gm$-representation with positive weights. The attracting locus is the zero section $V \subset T^* V$ and we need to show that there is a unique simple $\Gm$-monodromic local system on a $\Gm$-representation with positive weights. 

Let $\mathcal D= \mathcal D(\mathbb A^{d_i})$ denote the Weyl algebra equipped with the filtration $F_\bullet$ induced by the elliptic action on $\mathcal W_{S_i}$. The Hamiltonian action $\mathsf{H}$ induces a $\mathbb Z$-grading $\mathcal D = \bigoplus_{n \in \mathbb Z} \mathcal D_n$ where $\deg x_i > 0$ and $\deg \partial_i = - \deg x_i$. Note that  $\mathcal D_0 = \C\langle x_i\partial_i \rangle$ is a commutative subalgebra. Since the elliptic and Hamiltonian actions commute, the filtration and grading are compatible. Moreover, since $F_k\mathcal D$ is finite dimensional, so is $F_k \mathcal D\cap \mathcal D_n$ for every $n$. Suppose that $M$ is a simple monodromic $\mathcal D$-module supported on $V$. Let $M =\bigoplus_{n \in \mathbb Z} M_n$ be the grading on $M$ given by the Hamiltonian $\mathbb G_m$-action. Since $M$ is supported on $V$, we claim that $M_n=0$ for all $n$ sufficiently negative. To see this, note that since the associated graded of $M$ (with respect to the filtration $F_\bullet$) is supported on $V$, the generators of $\text{gr}_F\mathcal D$ with negative grading must act nilpotently. Since $\mathrm{gr}_F M$ is finitely generated, the result follows.  

Let $M_l \neq 0$ and $M_k =0$ for all $k<l$. Since $\partial_i$ has negative degree, $\partial_i m = 0$ for all $m \in M_l$. Thus, choosing a non-zero $m \in M_l$ defines a non-zero morphism $\mathcal{O}(V) = \mathcal{D} / (\partial_1,\partial_2, \dots ) \to M$ of $\mathcal D$-modules. Since $M$ is simple, this is an isomorphism. 


\end{proof}

\subsection{Duality and Hamiltonian Reduction}\label{sec:dualityetc}

In this subsection we will briefly discuss the duality functor on
quantizations and its application to the construction of the (partially
defined) functor $i^{*}$. We will concern ourselves only with a local
situation, as that is all that we need for this paper. In particular,
following the notation of \cite{BDMN}, Section 4, we assume $\mathfrak{X}=\text{Spec}(R)$ is a smooth affine $\mathbb{G}_{m}$-equivariant symplectic variety, and $I$ is the ideal in $R$ generated by the homogeneous elements of negative degree. The closed subvariety $C=\text{Spec}(R/I)$ is co-isotropic, and the symplectic variety $S= \mathrm{Spec} (R/I)^{\{I,\cdot\}}$ is obtained via Hamiltonian reduction as in Theorem~\ref{thm:coistropicreduction}. We let $\mathfrak{C}$ denote the formal
completion of $C$ in $\mathfrak{X}$. There is a (non-canonical)
equivariant isomorphism 
\[
\mathfrak{C}\iso S\times\widehat{T^{*}V}
\]
where $V$ is a vector space and $\widehat{T^{*}V}$ is the completion
of $T^{*}V$ along $V$. Under this isomorphism, we have $C \iso S\times V$. 

If $Q$ is an algebra then $\LMod{Q}$ denotes the category of left $Q$-modules and $\Lmod{Q}$ the category of finitely generated left $Q$-modules. We identify the category of right $Q$-modules with $\LMod{Q^{\opp}}$.

If $A$ denotes our given quantization of $R$, then let $J$ be
the left ideal generated by homogeneous elements of negative degree.
Let $B$ denote the Hamiltonian reduction algebra 
\[
B=\text{End}(A/J)^{\opp}.
\]
Then we have:

\begin{prop}\label{prop:completeAJ}
The ideal $J$ is generated by a graded vector space $\mathfrak{u}$ of dimension
$\text{dim}(V)$. Let $\widehat{A}$ denote the completion of $A$
along $J$. Then there is an equivariant isomorphism of algebras
\[
\widehat{A}\iso B\widehat{\otimes}\mathcal{\widehat{D}}
\]
where $\widehat{\mathcal{D}}$ denotes the completion of the Weyl
algebra $\mathcal{D}$ of $V$ along the left ideal generated by $\mathfrak{u}$.
Under this isomorphism, $\widehat{J}$ corresponds to $B\widehat{\otimes}\mathcal{\widehat{D}}\mathfrak{u}$
and $\widehat{A}/\widehat{J}$ corresponds to $B\widehat{\otimes}(\mathcal{\widehat{D}}/\widehat{\mathcal{D}}\mathfrak{u})$. 
\end{prop}

\begin{proof}
First, we note that in \cite[Section~4]{BDMN}, the algebra $\widehat{A}$ was defined to be the completion of $A$ along the two-sided ideal $K$, which is the preimage of $I$ under the quotient $A \to A / \hbar A \cong R$. But this is isomorphic, as topological algebras, to $\widehat{A}$ (as defined in the statement of the proposition) by \cite[Lemma~4.15(1)]{BDMN}. Then the first statement is Lemma~4.10 of \cite{BDMN}, as is the identification $\widehat{J} \iso B\widehat{\otimes}\mathcal{\widehat{D}}\mathfrak{u}$. The second statement is Proposition~4.9 of \cite{BDMN}. Finally, the proof of \cite[Proposition~4.12]{BDMN} explains that $\widehat{A}/\widehat{J}$ is mapped isomorphically to $B\widehat{\otimes}(\mathcal{\widehat{D}}/\widehat{\mathcal{D}}\mathfrak{u})$. 
\end{proof}

We require a better understanding of the relation between $\widehat{A}$-modules
and $A$-modules. Certainly, there is a flat morphism of algebras
$A\to\widehat{A}$, and the corresponding functor $M\to\widehat{A}\otimes_{A}M$,
which agrees with the naive completion when $M$ is finite over $A$.
In certain cases we can can actually recover $M$ from its completion.
To state the relevant result, recall the adjoint pair of functors
$\mathbb{H} \colon \LMod{A} \to \LMod{B}$ and $\mathbb{H}^{\perp} \colon \LMod{B} \to \LMod{A}$
given by 
\[
\mathbb{H}(M)=\Hom_{A}(A/J,M)
\]
and 
\[
\mathbb{H}^{\perp}(N)=A/J\otimes_{B}N .
\]
We have also the adjoint pair of functors 
\[
\widehat{\mathbb{H}}(M)=\Hom_{\widehat{A}}(\widehat{A}/\widehat{J},M)
\]
from $\LMod{\widehat{A}}$ to $\LMod{B}$ and
\[
\widehat{\mathbb{H}}^{\perp}(N)=\widehat{A}\otimes_{A}(A/J\otimes_{B}N)=\widehat{A}/\widehat{J}\otimes_{B}N
\]
from $\LMod{B}$ to $\LMod{\widehat{A}}$. We use the same
letters to denote the corresponding functors for $B_{\hbar}$-modules,
resp. $A_{\hbar}$ and $\widehat{A}_{\hbar}$-modules.

As explained in \cite[Proposition~4.14, Lemma~4.17]{BDMN}, we have isomorphisms
\begin{equation}\label{eq:RendeqB}
    \R \Hom_{\widehat{A}_{\hbar}}(\widehat{A}/\widehat{J}_{\hbar},\widehat{A}/\widehat{J}_{\hbar})\iso \R \Hom_{A_{\hbar}}(A/J_{\hbar},A/J_{\hbar}) \cong B^{\opp}_{\hbar}.
\end{equation}
From this we deduce:

\begin{prop}
Let $N$ be a finite $B_{\hbar}$-module, and let $M=\mathbb{H}^{\perp}(N)$.
Then there is a canonical isomorphism 
\[
N\iso\widehat{\mathbb{H}}\circ\widehat{\mathbb{H}}^{\perp}(N)
\]
and therefore a canonical isomorphism 
\[
M\iso\mathbb{H}^{\perp}\circ\widehat{\mathbb{H}}\circ\widehat{\mathbb{H}}^{\perp}(N).
\]
In particular, if $M_{1}$ and $M_{2}$ are two finite $A_{\hbar}$-modules,
both of which are in the image of $\mathbb{H}^{\perp}$, then $\widehat{M}_{1}\iso\widehat{M}_{2}$
implies $M_{1}\iso M_{2}$. 
\end{prop}

\begin{proof}
By \cite[Theorem~4.21]{BDMN}, there is an isomorphism $N\iso\mathbb{H}\circ\mathbb{H}^{\perp}(N)=\mathbb{H}(M)$.
So we only have to show that the canonical map $M\to\widehat{M}$
induces
\[
\mathbb{H}(M)\iso\widehat{\mathbb{H}}(\widehat{M}).
\]
As $M=A_{\hbar}/J_{\hbar} \otimes_{B_{\hbar}}N$ for a finite $B_{\hbar}$-module $N$, one reduces immediately by adjunction to showing 
\[
\mathbb{H}(A_{\hbar}/J_{\hbar})\iso\widehat{\mathbb{H}} (\widehat{A}_{\hbar}/\widehat{J}_{\hbar}),
\]
but this follows directly from equation \eqref{eq:RendeqB}. 
\end{proof}

\begin{remark}
By \cite[Theorem~4.21]{BDMN}, the image of $\Lmod{B_{\hbar}}$ under $\mathbb{H}^{\perp}$ is the full subcategory of $\Lmod{A_{\hbar}}$ supported on $C$. Moreover, $\Hamp$ is exact and hence extends to a functor between the derived categories associated to these abelian categories. 
\end{remark}

With this in hand, let us turn to duality functors. To save ourselves
from writing the same definition many times, let $Q$ denote any of
$A$, $\widehat{A}$, or $B$. As these are
algebras of finite homological dimension, we may define 
\[
\mathbb{D}(M)=\R \Hom_{Q_{\hbar}}(M,Q_{\hbar})[d]
\]
as a contravariant functor $D_{f.g.}^{b}(\LMod{Q_{\hbar}})\to D_{f.g.}^{b}(\LMod{Q^{\opp}_{\hbar}})$; where $2d = \dim \resol$ if $Q=A$ or $2d = \dim S$
if $Q=B$. Replacing $Q$ by $Q^{\text{opp }}$, we obtain
a functor $D_{f.g.}^{b}(\LMod{Q^{\opp}_{\hbar}})\to D_{f.g.}^{b}(\LMod{Q_{\hbar}})$,
also denoted by $\mathbb{D}$. Both of these functors are equivalences
of categories, and we have $\mathbb{D}\circ\mathbb{D}= \text{Id}$.
Note that $A^{\opp}$ is a quantization of $\mathfrak{X}$ (but with negative symplectic form) and $B^{\opp}$
is a quantization of $S$; so the above discussion, and the construction
of $\mathbb{H}$ and $\mathbb{H}^{\perp}$, applies verbatim to the
these algebras as well. 
\begin{thm}
Let $N$ be a finite $B_{\hbar}$-module. Then there is an isomorphism
\[
\mathbb{D}\circ\mathbb{H}^{\perp}(N)\iso\mathbb{H}^{\perp}\circ\mathbb{D}(N)
\]
in $D^b(\Lmod{A_{\hbar}})$. 
\end{thm}

\begin{proof}
We shall first show the isomorphism 
\begin{equation}\label{eq:AJdualiso}
  \mathbb{D}(A/J_{\hbar})\iso \R \Hom_{A_{\hbar}}(A_{\hbar}/J_{\hbar},A_{\hbar})\iso(J'\backslash A)_{\hbar}[-d]  
\end{equation}
where $J'$ is the right ideal of $A$ generated by $\mathfrak{u}$,
and $d=\text{dim}(V)$. To show this isomorphism, we first tackle
the completed version 
\begin{equation}
\R \Hom_{\widehat{A}_{\hbar}}(\widehat{A}_{\hbar}/\widehat{J}_{\hbar},\widehat{A}_{\hbar})\iso (\widehat{J'}\backslash\widehat{A})_{\hbar}[-d]\label{eq:completed-iso}
\end{equation}
For this we recall that under the isomorphism $\widehat{A} \cong B\widehat{\otimes}\widehat{\mathcal{D}}$,
$\widehat{A}/\widehat{J}$ corresponds to $B\widehat{\otimes}(\mathcal{\widehat{D}}/\widehat{\mathcal{D}}\mathfrak{u})$.
As in the proof of \cite[Proposition~4.14]{BDMN}, we can consider the Koszul complex $B\widehat{\otimes}K(\widehat{\mathcal{D}},\mathfrak{u})$ which is a finite free resolution of $\widehat{A}/\widehat{J}$. We have 
\[
H^i(\R \Hom_{\widehat{A}_{\hbar}}(\widehat{A}/\widehat{J}_{\hbar},\widehat{A}_{\hbar})) \iso H^i(\underline{\Hom}_{\widehat{A}_{\hbar}}(B_{\hbar}\widehat{\otimes}K(\widehat{\mathcal{D}},\mathfrak{u})_{\hbar},\widehat{A}_{\hbar}))
\]
\[
\iso H^i(\wedge^{\bullet}\mathfrak{u}^{*}\otimes\widehat{A}_{\hbar}),
\]
where we have used adjunction in the last step. The complex of right
modules so obtained is then the Koszul resolution for $(\widehat{J'}\backslash\widehat{A})_{\hbar}$,
placed in degree $d$; which implies the above isomorphism.

Now, to obtain the uncompleted version, we note that, as the duality
functor commutes with localization, the complex $\mathbb{D}(A_{\hbar}/J_{\hbar})$
is supported along $C$. Therefore, by \cite[Theorem~4.21]{BDMN}, each
$H^{i}(\mathbb{D}(A_{\hbar}/J_{\hbar}))$ is in the image of $\mathbb{H}^{\perp}$.
Therefore we may apply the previous proposition (for the algebra $A^{\opp}$)
to deduce that, since 
\[
H^{d}(\mathbb{D}(\widehat{A}_{\hbar}/\widehat{J}_{\hbar}))\iso (\widehat{J'}\backslash\widehat{A})_{\hbar}
\]
and
\[
H^{i}(\mathbb{D}(\widehat{A}_{\hbar}/\widehat{J}_{\hbar}))=0
\]
for all $i\neq d$, we must also have 
\[
H^{d}(\mathbb{D}(A_{\hbar}/J_{\hbar}))\iso (J'\backslash A)_{\hbar}
\]
and $H^{i}(\mathbb{D}(A/J_{\hbar}))=0$ for all $i\neq d$ as claimed. 

Now, we compute 
\begin{align*}
\mathbb{D}\circ\mathbb{H}^{\perp}(N)& =\R \Hom_{A_{\hbar}}(A_{\hbar}/J_{\hbar}\otimes_{B_{\hbar}}^\L N,A_{\hbar})[\dim (\mathfrak{X})/2] \\
& \iso \R \Hom_{B_{\hbar}}(N,\R \Hom_{A_{\hbar}}(A_{\hbar}/J_{\hbar},A_{\hbar}))[\dim (\mathfrak{X})/2] \\
& \iso \R \Hom_{B_{\hbar}}(N,J'_{\hbar}\backslash A_{\hbar})[\dim (\mathfrak{X})/2-d] \\
& \iso \R \Hom_{B_{\hbar}}(N,B_{\hbar})[\dim (S)/2]\otimes_{B_{\hbar}}^\L J_{\hbar}'\backslash A_{\hbar}=\mathbb{H}^{\perp}\circ\mathbb{D}(N)
\end{align*}
as desired. 
\end{proof}

\begin{corollary} There is an isomorphism of functors
    $$
    \mathbb{D}\circ\mathbb{H}^{\perp} \iso\mathbb{H}^{\perp}\circ\mathbb{D} \colon D^b_{f.g.}(\LMod{B_{\hbar}}) \to D^b_{f.g.}(\LMod{A_{\hbar}}).
    $$
\end{corollary}

Now we can define the desired functor $\mathbb{H}^{*}$. Recall (see \cite[Chapter~6]{KS}) that we have the category $\text{Pro}(\Lmod{A_{\hbar}})$
the category of pro-objects on $\Lmod{A_{\hbar}}$,
and its derived category $D(\text{Pro}(\Lmod{A_{\hbar}}))$. If $P$ is any projective object in $\LMod{A_{\hbar}}$, then, writing $P$ as an inductive limit of finite $A_{\hbar}$-modules,
the module $\Hom(P,A_{\hbar})$ is naturally an element of $\text{Pro}(\LMod{A_{\hbar}^{\opp}})$. Therefore
the duality $\mathbb{D}$ extends to a contravariant functor 
\[
\mathbb{D} \colon D(\LMod{A_{\hbar}})\to D(\text{Pro}(\Lmod{A_{\hbar}}))
\]
which is easily seen to be an equivalence of categories; indeed, its inverse is given by the cocontinuous extension of $\mathbb{D} \colon D(\Lmod{A_{\hbar}}) \to D(\LMod{A_{\hbar}})$
to $D(\text{Pro}(\Lmod{A_{\hbar}}))$; which we will
also denote by $\mathbb{D}$. Similarly, the functor $\mathbb{H}^{\perp}$
admits a cocontinious extension to a functor $\mathbb{H}^{\perp}:D(\text{Pro}(\Lmod{B_{\hbar}}))\to D(\text{Pro}(\Lmod{A_{\hbar}}))$. Then we have:

\begin{prop}\label{prop:adjunctionforHamperp}
Let $N\in D(\mathrm{Pro}(\Lmod{A_{\hbar}}))$. Then we
define 
\[
\mathbb{H}^{*}(N):=\mathbb{D}\mathbb{H}(\mathbb{D}N)\in D(\mathrm{Pro}(\Lmod{B_{\hbar}}))
\]
We have a functorial isomorphism 
\[
\R\mathrm{Hom}_{A_{\hbar}}(N,\mathbb{H}^{\perp}M)\iso \R\mathrm{Hom}_{B_{\hbar}}(\mathbb{H}^{*}N,M)
\]
for all $N\in D(\mathrm{Pro}(\Lmod{A_{\hbar}}))$ and
$M\in D(\mathrm{Pro}(\Lmod{B_{\hbar}}))$. 
\end{prop}

\begin{proof}
Let $N=\mathbb{D}N'$ and $M=\mathbb{D}M'$. Then 
\[
\R\Hom_{A_{\hbar}}(N,\mathbb{H}^{\perp}M)\iso \R\Hom_{A_{\hbar}}(\mathbb{D}N',\mathbb{H}^{\perp}\mathbb{D}M')
\]
\[
\iso \R\Hom_{A_{\hbar}}(\mathbb{D}N',\mathbb{D}\mathbb{H}^{\perp}M')\iso \R\Hom_{A^{\opp}_{\hbar}}(\mathbb{H}^{\perp}M',N')
\]
\[
\iso \R\Hom_{B^{opp}_{\hbar}}(M',\mathbb{H}N')\iso \R\Hom_{B_{\hbar}}(\mathbb{H}^{*}N,M)
\]
as claimed. 
\end{proof}

\subsection{Preservation of Holonomicity}\label{sec:preserveholonomic}

Let $U\subset U'$ be open subsets of $\resol$ which are unions of coisotropic cells and write $Z = U' \setminus U$. We form the usual diagram $Z \stackrel{i}{\hookrightarrow} U' \stackrel{j}{\hookleftarrow} U$, as in \eqref{eq:poencloseddiagram}. Denote by $D_{\hol}^{b}(\QCoh{\cW_{\resol}})$ the derived category of quasi-coherent $\cW$-modules with holonomic cohomology. The following key result is crucial in proving our main result that there exists a holonomic local generator in $D^b(\Good{\cW})$. The proof is given in the next subsection.

\begin{thm}\label{thm:holonomicpushforward}
Let $\ms{M}\in D_{\hol}^{b}(\QCoh{\cW_{U}})$. Then $j_{*} \ms{M} \in D_{\hol}^{b}(\QCoh{\cW_{U'}})$.
Equivalently, if $\ms{N} \in D_{\hol}^{b}(\QCoh{\cW_{U'}})$,
then $\R i^{!} \ms{N} \in D_{\hol}^{b}(\QCoh{\cW_{Z}})$. 
\end{thm}

As a consequence of Theorem~\ref{thm:holonomicpushforward}, we may define
\begin{align*}
    j_! & = \mathbb{D} \circ \R j_* \circ \mathbb{D} \colon D_{\hol}^b(\QCoh{\cW_U}) \to D_{\hol}^b(\QCoh{\cW_{U'}}), \\
i^* & = \mathbb{D} \circ \R i^! \circ \mathbb{D} \colon D_{\hol}^b(\QCoh{\cW_{U'}}) \to D_{\hol,Z}^b(\QCoh{\cW_{U'}}).
\end{align*}
and get a recollment pattern
    \[
    \begin{tikzcd}
    D^b_{\hol,Z}(\QCoh{\cW_{U'}}) \ar[rr,"{i_*}"] & & \ar[ll,"{\R i^!}"',bend right=20] \ar[ll,"{i^*}",bend left=15] D^b_{\hol}(\QCoh{\cW_{U'}}) \ar[rr,"{j^*}"] & &  D^b_{\hol}(\QCoh{\cW_{U}})  \ar[ll,"{\R j_*}"',bend right=20] \ar[ll,"{j_!}",bend left=15].    
    \end{tikzcd}
    \]
This means that $i^*$ is left adjoint (resp. $\R i^!$ right adjoint) to $i_*$ and $j_!$ left adjoint (resp. $\R j_*$ right adjoint) to $j^*$. Moreover, $j^* \circ i_* = 0, i^* \circ j^! = 0$ and $\R i^! \circ \R j_* = 0$. These functors give rise to the usual distinguished triangles 
\[
i_! \circ \R i^! (\ms{M}) \to \ms{M} \to \R j_* \circ j^* (\ms{M}) \stackrel{[1]}{\rightarrow}
\]
\[
i_* \circ i^* (\ms{M}) \to \ms{M} \to j_! \circ j^! (\ms{M}) \stackrel{[1]}{\rightarrow}, 
\]
noting that $i_* = i_!$ and $j^! = j^*$. Finally, the functors $i_*,\R j_*$ and $j_!$ are fully faithful. 

\subsection{Proof of Theorem~\ref{thm:holonomicpushforward}}

The fact that the two statements of the theorem are equivalent follows from the triangle \eqref{eq:openclosedtriangle}:
\[
i_{*} \R i^{!} \ms{N} \to \ms{N} \to j_{*}(\ms{N}|_{U}),
\]
as well as the fact that every holonomic module in $\QCoh{\cW_{U}}$
admits a holonomic extension to $\QCoh{\cW_{U'}}$ by \cite[Corollary~3.33]{BDMN}. 

To proceed, note that by induction on the number of cells, we may
suppose that $Z = U' \setminus U$ is a single cell, called $C$. We shall
prove the second statement in this context. As the statement is local,
we can suppose we are in the affine situation, and $\Gamma(\resol,\cW_{\resol})=A_{\hbar}$.
Let $\widehat{A}_{\hbar}$ denote the completion of $A_{\hbar}$ along
the ideal $J$ (notation as in Section~\ref{sec:dualityetc}). By Proposition~\ref{prop:completeAJ}, we have the isomorphism 
\[
\widehat{A}=\widehat{\mathcal{D}}\widehat{\otimes}B
\]
where $\widehat{\mathcal{D}}$ is the standard quantization of $\widehat{T^{*}V}$
and $B$ is the quantization of $S$, the coisotropic reduction of
$C$. It is explained in \cite[Section 2.6]{BDMN} that the fact that $S$ has a single fixed point for the elliptic $\mathbb{G}_{m}$-action implies that the algebra $B_{\hbar}$ is isomorphic to the $\hbar$-Weyl algebra $\mathcal{D}_{\hbar}$, of rank $m = (1/2)\dim S$. 

\begin{lem}\label{lem:AtoAhatBholo}
    Let $M$ be a good $A_{\hbar}$-module. Then 
    \[
    \R i^! M [n] \cong J_{\hbar}' \backslash A_{\hbar} \o_{A_{\hbar}}^{\mathbb{L}} M \cong \widehat{J}_{\hbar}' \backslash \widehat{A}_{\hbar} \o_{\widehat{A}_{\hbar}}^{\mathbb{L}} \widehat{M},
    \]
    where $n = \dim V$ and $\widehat{M}=\widehat{A}_{\hbar}\otimes_{A_{\hbar}} M$. 
\end{lem}

\begin{proof}
    Recall that $i^! M = \Ham \circ \Gamma_C(M) = \Hom_{A_{\hbar}}(A_{\hbar} / J_{\hbar}, \Gamma_C(M))$, where $\Gamma_C(M)$ is the largest submodule of $M$ supported on $C$. If $\phi \in \Hom_{A_{\hbar}}(A_{\hbar} / J_{\hbar}, M)$ then $\mathrm{Im} \, \phi$ is a quotient of $A_{\hbar} / J_{\hbar}$ and hence supported on $C$. Thus, $\mathrm{Im} \, \phi \subset \Gamma_C(M)$ meaning that $i^! M = \Hom_{A_{\hbar}}(A_{\hbar} / J_{\hbar}, M)$. 

    Passing to derived functors, 
    \begin{align*}
        \R i^! M & = \R \Hom_{A_{\hbar}}(A_{\hbar} / J_{\hbar}, M) \\
        & \iso \R \Hom_{A_{\hbar}}(A_{\hbar} / J_{\hbar}, A_{\hbar}) \o_{A_{\hbar}}^{\mathbb{L}} M \iso J_{\hbar}' \backslash A_{\hbar} \o_{A_{\hbar}}^{\mathbb{L}} M [-n], 
    \end{align*}
 where we applied the isomorphism \eqref{eq:AJdualiso}. This establishes the first isomorphism. 

The proof of \cite[Theorem~4.15]{BDMN} shows that $J_{\hbar}' \backslash A_{\hbar} \iso \widehat{J}_{\hbar}' \backslash \widehat{A}_{\hbar}$ as $(B_{\hbar},A_{\hbar})$-bimodules. Since $\widehat{A}_{\hbar}$ is flat over $A_{\hbar}$, 
\[
 \widehat{J}_{\hbar}' \backslash \widehat{A}_{\hbar} \o_{\widehat{A}_{\hbar}}^{\mathbb{L}} \widehat{M} \cong   \widehat{J}_{\hbar}' \backslash \widehat{A}_{\hbar} \o_{\widehat{A}_{\hbar}}^{\mathbb{L}} \widehat{A}_{\hbar} \o_{A_{\hbar}}^{\mathbb{L}} M \cong \widehat{J}_{\hbar}' \backslash \widehat{A}_{\hbar} \o_{A_{\hbar}}^{\mathbb{L}} M,
 \]
proving the second isomorphism. 
\end{proof}

Again, let $Q$ be any one of $A,\widehat{A}$ or $B$. Recall that if $M$ is a finite $Q_{\hbar}$-module then the \textit{grade} $j(M)$ of $M$ is the smallest $j \ge 0$ such that $\Ext^j_{Q_{\hbar}}(M,Q_{\hbar}) \neq 0$. Since $Q_{\hbar}$ is Auslander regular and $Q / \hbar Q$ is Cohen-Macaulay, 
\[
j(M) + \dim \Supp \, M = \dim \spec (Q / \hbar Q). 
\]
Moreover, by Gabber's Theorem, $\dim \Supp M \, \le (1/2) \dim \spec (Q / \hbar Q)$. Thus, 
\[
j(M) \le (1/2) \dim \spec (Q / \hbar Q)
\]
and $j(M) = (1/2) \dim \spec (Q / \hbar Q)$ if and only if $M$ is holonomic. 


Since $\widehat{A}_{\hbar}$ is flat over $A_{\hbar}$, we have 
\[
\Ext^{i}(M,A_{\hbar})\otimes_{A_{\hbar}} \widehat{A}_{\hbar} \iso \Ext^{i}(\widehat{A}_{\hbar} \otimes_{A_{\hbar}} M,\widehat{A}_{\hbar}),
\]
which shows that if $M$ is a holonomic $A_{\hbar}$-module then $\widehat{A}_{\hbar} \otimes_{A_{\hbar}} M$ is a holonomic $\widehat{A}_{\hbar}$-module. Therefore, by Lemma~\ref{lem:AtoAhatBholo}, it suffices to prove that if $N$ is a holonomic $\widehat{A}_{\hbar}$-module then the cohomology groups 
\[
H^k\left( \widehat{J}_{\hbar}' \backslash \widehat{A}_{\hbar} \o_{\widehat{A}_{\hbar}}^{\mathbb{L}} N \right)
\]
are holonomic $B_{\hbar}$-modules. 

Choose local coordinates $\{x_{1},\dots,x_{n}\}$ for $V$. Then the
algebra $(\widehat{\mathcal{D}}_{\hbar})^{\mathbb{G}_{m}}$ can be
explicitly described as follows: let $\deg (\partial_{i})=a_{i} > 0$
and $\deg (x_{i})=-b_{i} < 0$, so that $a_{i}-b_{i}= \ell > 0$ because $[\partial_i,x_i] = \hbar^{\ell}$. Let $A=(a_{1},\dots,a_{n})$
and $B=(b_{1},\dots b_{n})$ be the associated multi-indices. Then
$(\widehat{\mathcal{D}}_{\hbar})^{\mathbb{G}_{m}}$ is the collection
of sums 
\[
\sum_{I,K}a_{I,K}(\hbar^{|B|}x)^{I}(\hbar^{-|A|}\partial)^{K}=\sum_{I,K}a_{I,K}\hbar^{|B\cdot I|-|A\cdot K|}x^{I}\partial^{K},
\]
where $I,K$ are multi-indices, $a_{I,K}\in\C$, and we set
$|B\cdot I|={\displaystyle \sum_{j=1}^{n}b_{j}i_{j}}$, and similarly
for $|A\cdot K|$. Relabeling $\hbar^{b_{i}}x_{i}$ as $x_{i}$ and $\hbar^{-a_{i}}\partial_{i}$
as $\partial_{i}$, we see that $(\widehat{\mathcal{D}}_{\hbar})^{\mathbb{G}_{m}}$
is the algebra of sums of the form 
\[
\sum_{I,K}a_{I,K}x^{I}\partial^{K}
\]
where the sum ranges over $\{I,K\}$ so that $|B\cdot I|-|A\cdot K|\geq M$
for some $M\in\mathbb{Z}$. 

Choosing coordinates on $S$, we have that $(B_{\hbar})^{\mathbb{G}_{m}}$
is the usual Weyl algebra 
\[
\C \langle y_{1},\dots y_{m},\partial_{1},\dots,\partial_{m} \rangle
\]
Thus $(\widehat{A}_{\hbar})^{\mathbb{G}_{m}}$ is isomorphic to the
algebra $\widehat{D}_{n,m}$, which is explicitly given as the algebra
of formal sums of the form 
\[
\sum_{I,J,K,L}a_{I,J,K,L}x^{I}y^{J}\partial_{x}^{K}\partial_{y}^{L}
\]
where $a_{I,J,K,L}\in\mathbb{C}$ and the sum ranges over $\{I,K\}$
so that $|B\cdot I|-|A\cdot K|\geq M$ for some $M\in\mathbb{Z}$ and $J,L \ge 0$. Good modules over $\widehat{A}_{\hbar}$ are equivalent to finite
modules over $\widehat{D}_{n,m}$. 

Now, let us consider the $A$-weighted $V$-filtration of $\widehat{D}_{n,m}$
along $\{x_{1},\dots,x_{n}\}$; explicitly, we set 
\[
V_{t}(\widehat{D}_{n,m})=\left\{\sum_{I,J,K,L}a_{I,J,K,L}x^{I}y^{J}\partial_{x}^{K}\partial_{y}^{L}:|A\cdot K|-|A\cdot I|\leq t \right\}, \quad \forall \, t \in \Z. 
\]
This is clearly a separated filtration of $\widehat{D}_{n,m}$, let
us show it is also exhaustive: if 
\[ 
P={\displaystyle \sum_{I,J,K,L}a_{I,J,K,L}x^{I}y^{J}\partial_{x}^{K}\partial_{y}^{L}}
\]
is an arbitrary element of $\widehat{D}_{n,m}$, then there is $M\in\mathbb{Z}$
so that 
\[
|B\cdot I|-|A\cdot K|\geq M
\]
i.e., 
\[
|A\cdot K|\leq|B\cdot I|-M\leq|A\cdot I|-M,
\]
where we used that $A - B = (\ell,\dots, \ell)$. Then  
\[
|A\cdot K|-|A\cdot I|\leq-M
\]
and so we have $P\in V_{-M}(\widehat{D}_{n,m})$. 

Denote by $\widehat{D}_{n,m,V}$ the completion of $\widehat{D}_{n,m}$
along this $V$-filtration. By part (1) of Lemma~\ref{lem:ReesV} below, the algebra $\widehat{D}_{n,m,V}$ is flat over $\widehat{D}_{n,m}$. Furthermore,
by part (2) of Lemma~\ref{lem:ReesV}, we have $\text{gr}(\widehat{D}_{n,m,V}) \cong D_{n+m}$, a Weyl algebra which is a ring of global dimension $n+m$. Therefore, by \cite[Theorem~3.4]{HO},
$\widehat{D}_{n,m,V}$ has global dimension $\leq n+m$;
by considering the module $\mathbb{C}[y_{1},\dots y_{m}][[x_{1},\dots,x_{n}]]$
we see that the global dimension is exactly $n+m$. We therefore have
a notion of holonomic module over $\widehat{D}_{n,m,V}$, namely those
modules $M$ for which $\text{Ext}^{i}(M,\widehat{D}_{n,m,V})=0$ for $i<n+m$;
i.e., the grade number $j(M)$ is $n+m$.

Now, let $N$ be a holonomic module over $\widehat{D}_{n,m}$, equipped
with a good filtration with respect to $V$. Then we claim that $\widehat{D}_{n,m,V}\otimes_{\widehat{D}_{n,m}}N \cong \widehat{N}_V$
is holonomic over $\widehat{D}_{n,m,V}$, where $\widehat{N}_V$ is the completion of $N$ with respect to the good filtration. Indeed, using the flatness
of $\widehat{D}_{n,m,V}$ over $\widehat{D}_{n,m}$, we have 
\[
\Ext^{i}(N,\widehat{D}_{n,m})\otimes_{\widehat{D}_{n,m}}\widehat{D}_{n,m,V} \iso \Ext^{i}(\widehat{D}_{n,m,V}\otimes_{\widehat{D}_{n,m}}N,\widehat{D}_{n,m,V})
\]
from which the result follows. 

Recall that, by the results of\footnote{Here, we use the fact that $\widehat{D}_{n,m,V}$ is a Zariskian ring;
this is why we had to complete along $V$.} \cite[Chapter~3, Section~2.5]{HO}, we have that $j(M)=j(\text{gr}(M))$
for any module $\widehat{D}_{n,m,V}$-module $M$ equipped with a
good filtration. In particular, $\text{gr}(M)$ is a holonomic $D_{n+m}$-module
in the usual sense. Therefore, if $N$ is a holonomic module over
$\widehat{D}_{n,m}$, equipped with a good filtration with respect
to $V$, then $\text{gr}(N)=\text{gr}(\widehat{N}_V)$ is holonomic
over $D_{n+m}$. 

Let $\mathcal{R}(\widehat{D}_{n,m})$ be the Rees algebra of $\widehat{D}_{n,m}$
with respect to the $V$-filtration; denote the Rees parameter by
$T$. Let $\mathcal{R}(D_{m})=D_{m}[T]$ denote the Rees algebra of
$D_{m}$ with respect to the trivial filtration where all generators have degree zero. Let $\hat{x}_i = T^{-a_i} x_i \in \mathcal{R}(\widehat{D}_{n,m})_{-a_i}$. Then we can define
a $(\mathcal{R}(D_{m}),\mathcal{R}(\widehat{D}_{n,m}))$-bimodule,
denoted $i^{!}(\mathcal{R}(\widehat{D}_{n,m}))$, via
\[
i^{!}(\mathcal{R}(\widehat{D}_{n,m}))=(\hat{x}_{1},\dots, \hat{x}_{n})\cdot\mathcal{R}(\widehat{D}_{n,m})\backslash\mathcal{R}(\widehat{D}_{n,m})
\]
and we define 
\[
i^{!}M:=i^{!}(\mathcal{R}(\widehat{D}_{n,m}))\otimes_{\mathcal{R}(\widehat{D}_{n,m})}^{\L}M
\]
for any $\mathcal{R}(\widehat{D}_{n,m})$-module $M$. One verifies
readily that 
\[
i^{!}(M)\otimes_{\mathbb{C}[T]}^{\L}(\mathbb{C}[T]/T)\cong i_{0}^{!}(M\otimes_{\mathbb{C}[T]}^{\L}(\mathbb{C}[T]/T))
\]
where $i_{0}^{!}:D^{b}(\LMod{D_{n+m}})\to D^{b}(\LMod{D_{m}})$
is the usual pullback functor; and 
\[
i^{!}(M)\otimes_{\mathbb{C}[T]}^{\L}(\mathbb{C}[T]/(T-1))\cong i_{1}^{!}(M\otimes_{\mathbb{C}[T]}^{\L}(\mathbb{C}[T]/(T-1)))
\]
where $i_{1}^{!} \colon D^{b}(\LMod{\widehat{D}_{n+m}})\to D^{b}(\LMod{D_{m}})$ is the pullback we are considering in this proof. Suppose $N$ is
holonomic over $\widehat{D}_{n,m}$. Equip it with a good filtration
and set $M=\mathcal{R}(N)$. Noting that $M$ is torsion-free over $T$,  
\[
M\otimes_{\mathbb{C}[T]}^{\L}(\mathbb{C}[T]/T)=\text{gr}(N)
\]
is holonomic over $D_{n+m}$, and so 
\[
H^{k}(i^{!}(M)\otimes_{\mathbb{C}[T]}^{\L}(\mathbb{C}[T]/T)) = H^k(i_0^! (\gr (N)))
\]
is holonomic over $D_{m}$, for all $k$. Since $\C[T] / (T)$ is resolved over $\C[T]$ by a two-term complex, there is a two-step filtration of $H^{k}(i^{!}(M)\otimes_{\mathbb{C}[T]}^{\L}(\mathbb{C}[T]/T))$ coming from the filtration of the associated bicomplex. The first term of this filtration is $H^{k}(i^{!}M)/T$. Therefore, there is a canonical
injection 
\[
H^{k}(i^{!}M)/T\to H^{k}(i^{!}(M)\otimes_{\mathbb{C}[T]}^{\L}(\mathbb{C}[T]/T)),
\]
which implies that $H^{k}(i^{!}M)/T$ is holonomic over $D_{m}$. On the
other hand, we have that 
\[
M\otimes_{\mathbb{C}[T]}^{\L}(\mathbb{C}[T]/(T-1))\cong N
\]
and 
\[
H^{k}(i^{!}M\otimes_{\mathbb{C}[T]}^{\L}(\mathbb{C}[T]/(T-1))\cong H^{k}(i^{!}M)/(T-1)\cong H^{k}(i_{1}^{!}N)
\]
Thus, the filtered module $H^{k}(i_{1}^{!}N)$ has an associated graded
which is holonomic over $D_{m}$. As holonomic modules are finite
length, we see that this filtration has $V_{s}(H^{k}(i_{1}^{!}N))\subsetneq V_{s+1}(H^{k}(i_{1}^{!}N))$ for only finitely many degrees $s$. We now claim that this filtration is bounded below. In fact, from the definition of $i^{!}M$, we can
compute it by choosing a finite projective resolution $\mathcal{P}^{\idot}$
of $M$ (in the category of graded $\mathcal{R}(\widehat{D}_{n,m})$-modules)
and then looking at the complex $\mathcal{P}^{\idot}/(\hat{x}_{1},\dots ,\hat{x}_{n})\mathcal{P}^{\idot}$. But each term $\mathcal{P}^{j}/(\hat{x}_{1},\dots, \hat{x}_{n})\mathcal{P}^{j}$ has bounded below grading; to see this it suffices to show that $\mathcal{R}(\widehat{D}_{n,m})/(\hat{x}_{1},\dots,\hat{x}_{n})\mathcal{R}(\widehat{D}_{n,m})$
has bounded below grading; and indeed its grading is bounded below
by $-|A|$ (compare the proof of the lemma below). 

Therefore $H^{k}(i_{1}^{!}N)$ possesses a finite filtration by $D_{m}$-modules
whose associated graded is holonomic over $D_{m}$. Hence $H^{k}(i_{1}^{!}N)$
is itself holonomic over $D_{m}$; this is what we wanted to prove. 

In the course of the above proof we needed the following results. 

\begin{lem}\label{lem:ReesV}
\begin{enumerate}
    \item[1)] The ring $V_{0}(\widehat{D}_{n,m})$ is (left) Noetherian. The
Rees ring $\mathcal{R}_{V}(\widehat{D}_{n,m})$ is (left) Noetherian
in the graded sense. 
\item[2)] We have $\gr_{V}(\widehat{D}_{n,m}) \cong D_{n+m}$, equipped
with the grading satisfying $\deg (x_{i})=-a_{i}$, $\deg (\partial_{x_{i}})=a_{i}$,
and $\deg (y_{i})=\deg (\partial_{y_{i}})=0$. 
\end{enumerate}
\end{lem}

\begin{proof}
We start with part $1)$. We can consider the $A$-weighted $V$-filtration
on $\widehat{A}=\widehat{\mathcal{D}}\widehat{\otimes}B$, defined
in the analogous way to the above; then $V_{0}(\widehat{D}_{n,m})=(V_{0}(\widehat{A})_{\hbar})^{\mathbb{G}_{m}}$.
So it suffices to show that $V_{0}(\widehat{A})$ is Noetherian. As
$V_{0}(\widehat{A})$ is $\hbar$-torsion-free and $\hbar$-adically complete,
it suffices to show that $V_{0}(\widehat{A})/\hbar$ is Noetherian (see e.g., \cite[Theorem~1.2.5]{KSDQ} for a more general statement). 

To see this, we start by noting that $\widehat{A}/\hbar$ is the completion of the polynomial ring 
\[
R=\mathbb{C}[x_{1},\dots,x_{n},y_{1},\dots,y_{m},\xi_{x_{1}},\dots,\xi_{x_{n}},\xi_{y_{1}},\dots,\xi_{y_{m}}]
\]
 along the $(x_{1},\dots,x_{n})$-adic filtration. This polynomial
ring is equipped with an algebraic $\mathbb{G}_{m}$ action corresponding
to the grading where $\text{deg}(x_{i})=-b_{i}$, $\text{deg}(\xi_{x_{i}})=a_{i}$,
and $\text{deg}(y_{j})=\text{deg}(\xi_{y_{j}})=0$. Therefore $\widehat{A}/\hbar$
is equipped with the corresponding topological $\mathbb{G}_{m}$-action,
and $V_{0}(\widehat{A})/ \hbar \subset\widehat{A}/\hbar$ is the subring consisting
of convergent sums of elements of weight $\leq 0$. 

Now, the basic theory of $\mathbb{G}_{m}$-actions on finitely generated
algebras ensures that the ring $R_{0}$ of degree $0$ elements is
finitely generated over $\mathbb{C}$, each $R_{0}$-module $R_{i}$
is finitely generated. The ring $R_{\leq 0}$ of elements of degree
$\leq 0$ is also finitely generated over $\C$ since it can be identified with $R[t]^{\Gm}$, where $t$ has weight one. We now claim that $V_{0}(\widehat{A})/\hbar$ is the completion of $R_{\leq 0}$
along the $\mathcal{I}$-adic filtration, where $\mathcal{I}=R_{\leq 0}\cap(x_{1},\dots,x_{n})$;
this proves that $V_{0}(\widehat{A})/\hbar$ is Noetherian as needed.

To see this note that there is an obvious map
\[
R_{\leq0}\to V_{0}(\widehat{A})/\hbar
\]
which extends to map from the $\mathcal{I}$-adic completion $\widehat{R}_{\leq0}$,
to $V_{0}(\widehat{A})/\hbar$. It is straightforward to see that this
map is injective. To see the subjectivity, consider a term 
\[
P=\sum_{I,J,K,L}a_{I,J,K,L}x^{I}y^{J}\xi_{x}^{K}\xi_{y}^{L}\in V_{0}(\widehat{A})/\hbar
\]
Choose some $m>0$. We claim that the set of terms $a_{I,J,K,L}x^{I}y^{J}\xi_{x}^{K}\xi_{y}^{L}$
in $P$ which are not in $\mathcal{I}^{m}\cdot R_{\leq 0}$, is finite.
If this is so, then we let $P_{m}$ be the (finite) sum of the terms
in $P$ which are in $\mathcal{I}^{m}\cdot R_{\leq 0}$ but not in
$\mathcal{I}{}^{m+1}\cdot R_{\leq 0}$. We can write 
\[
P=\sum_{m=0}^{\infty}P_{m}
\]
which expresses $P$ as an element of $\widehat{R}_{\leq0}$. 

To prove the claim, observe that $a_{I,J,K,L}x^{I}y^{J}\xi_{x}^{K}\xi_{y}^{L}$
is contained in $\mathcal{I}^{m}\cdot R_{\leq0}$ whenever $|I|\geq|A|\cdot m$.
So if this term in not in $\mathcal{I}^{m}\cdot R_{\leq0}$, we must
have $|I|<|A|\cdot m$. Since $P\in V_{0}(\widehat{A})/ \hbar$, we have
$|A\cdot K|\leq|A\cdot I|<|A|^{2}m$ so there are only finitely many
possible $I$ and $K$, which is what we needed. 

To deduce the statement about the Rees ring, we again work with $\widehat{A}$
instead of $\widehat{D}_{n,m}$. Note that $\mathcal{R}_{V}(\widehat{A})/\hbar \iso \mathcal{R}_{V}(\widehat{A}/\hbar)$.
It follows from the previous part that the $i$th graded piece $\mathcal{R}_{V}(\widehat{A}/\hbar)_{i}$
is the completion of $R_{\leq i}$ along the $\mathcal{I}$-adic filtration
(as an $R_{\leq0}$-module). Therefore each such module is finitely
generated and the entire algebra is finitely generated over $\widehat{R}_{\leq0}$
(as these facts are true for the algebra $R$ before completion);
and therefore Noetherian. But now the analogous statements for $\mathcal{R}_{V}(\widehat{A})$ follow from the complete Nakayama lemma, using the fact that each
graded piece $\mathcal{R}_{V}(\widehat{A})_{i}$ is $\hbar$-adically
complete. 

2) There is clearly a map from the graded ring $D_{n+m}$ to $\widehat{D}_{n,m}$,
which respects the filtrations on both sides. So we have to show that
\[
V_{i}(D_{n+m})/V_{i-1}(D_{n+m})\tilde{\to}V_{i}(\widehat{D}_{n,m})/V_{i-1}(\widehat{D}_{n,m}).
\]
Let 
\[
P=\sum_{I,J,K,L}a_{I,J,K,L}x^{I}y^{J}\partial_{x}^{K}\partial_{y}^{L}\in V_{i}(\widehat{D}_{n,m}).
\]
Then, for all terms $a_{I,J,K,L}x^{I}y^{J}\partial_{x}^{K}\partial_{y}^{L}$
that occur in $P$, we have $|B\cdot I|-|A\cdot K|\geq M$ for some
fixed $M\in\mathbb{Z}$. This forces 
\[
|A\cdot K|\leq|B\cdot I|-M
\]
so that 
\[
|A\cdot K|-|A\cdot I|\leq|(B-A)I|-M.
\]
As $B-A$ is a multi-index with strictly negative values, we see that
as $I$ increases we must have $|A\cdot K|-|A\cdot I|$ decreases.
Therefore, all but finitely many terms in $P$ are contained in $V_{i-1}(\widehat{D}_{n,m})$,
and the quotient consists of finite sums 
\[
\sum_{I,J,K,L}a_{I,J,K,L}x^{I}y^{J}\partial_{x}^{K}\partial_{y}^{L}
\]
with $|A\cdot K|-|A\cdot I|=i$, which is exactly $V_{i}(D_{n+m})/V_{i-1}(D_{n+m})$
as required. 
\end{proof}

\subsection{Extensions of Modules}\label{sec:externsionsympresmod}

Now we wish to specialize the situation: we suppose from here on in that $\resol$ is a bionic symplectic resolution of singularities. 
Thus we have a proper birational $\mathsf{T}$-equivariant morphism  $\mu: \resol \to X:= \resol^{\on{aff}}$, where $X$ is an affine symplectic variety, conic with respect to the elliptic action. Fix $r \in I$ and let $U(r)$ be an affine $\mathsf{T}$-stable open subset of $\resol$ containing the fixed point $y_r$. Then $U(r)$ contains both $\Lag_r$ and $C_r$. Let $\delta_{\Lag_r}$ denote the holonomic simple module on $U(r)$ constructed in Proposition~\ref{prop:uniquesimple}.

\begin{lemma}\label{lem:deltaboundedweights}
The weights of the  Hamiltonian $\Gm$-action on $\Gamma(U(r),\delta_{\Lag_r})^{\Gm}$ are bounded above.
\end{lemma}

\begin{proof}
By definition, $\delta_{\Lag_r}$ is supported on the attracting locus of the Hamiltonian $\Gm$-action on $U(r)$. Therefore, one can argue just as in the proof of Proposition~\ref{prop:uniquesimple}. Namely, $M := \Gamma(U(r),\delta_{\Lag_i})^{\Gm}$ has a good filtration whose associated graded is a $\Gamma(U(r),\mc{O})$-module supported on $\Lag_r$. It follows that the weights of $\gr M$ are bounded above. The same must therefore hold for $M$. 
\end{proof}

Then we have: 

\begin{prop}\label{simple extension prop}
The $\cW_{U(r)}$-module $\delta_{\Lag_r}$ admits an extension to a holonomic $\cW_{\resol}$-module which is set-theoretically supported on $\Lag$.

\end{prop}

\begin{proof}
Since we have assumed that $\resol$ is a symplectic resolution, it is known \cite{BLPWAst} that $\Good{\cW}$ is equivalent to $\Good{\cW'}$  where $\cW'$ is a quantization for which localization holds. Furthermore, this equivalence is given by twisting by a line bundle; so that the object $\delta_{\Lag_i}$ constructed above is sent to the corresponding module $\delta'_{\Lag_i}$ over $\cW'$. In addition, it follows that this twisting preserves the support of a coherent module. Thus to prove the proposition we can assume that localization holds for $\cW$.

Now, let $j \colon U(r) \to \resol$ be the inclusion. By \cite[Corollary~3.33]{BDMN}, there is a  coherent holonomic subsheaf $\ms{M}$ of $j_* \delta_{\Lag_i}$ whose restriction to $U$ is equal to $\delta_{\Lag_i}$, and which has no submodule supported on the complement of $U$ (see \cite[Proposition~3.25]{BDMN} for this statement). Now, let us momentarily denote by $j_{\idot}$ the naive push-forward of sheaves from $U$ to $\resol$ (this is, in general, different from the functor $j_{*}$ on quasicoherent $\cW$-modules). Since localization holds and $\ms{M}$ is simple, $\Gamma(\resol, \ms{M})^{\Gm}$ is a (non-zero) simple module over $\Gamma(\resol,\cW)^{\Gm}$. Therefore, the inclusion $\ms{M} \hookrightarrow \delta_{\Lag_i}$ gives an injection 
\[
\Gamma(\resol, \ms{M})^{\Gm} \to \Gamma(\resol, j_{\idot}j^* \ms{M})^{\Gm} = \Gamma(U(r), \delta_{\Lag_i})^{\Gm}.
\]
However, the weights of the Hamiltonian $\Gm$-action on the object on the right are bounded above by Lemma~\ref{lem:deltaboundedweights}. Thus the same is true of $\Gamma(\resol,\ms{M})^{\Gm}$. 

We denote the attracting set in $X$ with respect to the Hamiltonian action by $\mathfrak{n}$. Then it follows from the properness of $\mu$ that $\mu^{-1}(\mathfrak{n}) = \Lag$. Moreover, we have shown that $\Gamma(\resol,\ms{M})^{\Gm}$ is supported on $\mathfrak{n}$. But, applying the localization functor to $\Gamma(\resol,\ms{M})^{\Gm}$, it follows that $\ms{M}$ is supported on $\mu^{-1}(\mathfrak{n}) = \Lag$ and we are done.    
\end{proof}




\begin{exmp} \label{exmp:typeAholo}

Continue with the notation of \ref{exmp:Kleinian}.  Quantizations for $X_n$ have been explicitly studied in \cite{Ku}, and using his work one sees that for a generic choice of quantization, the simple extension $\mathcal L_i$ of the unique simple module supported on the Lagrangian $\Lag_i$ will be supported on $\bigcup_{0\leq j\leq i} \overline{\Lag}_j$, and so the multiplicity of $\mathcal L = \bigoplus_{i=0}^{n} \mathcal L_i$ along $C_i$ is $n-i+1$. For special quantizations however, there may be different multiplicities. Note that one could chose other extensions of the simple module on $C_i$: for example the sheaves corresponding to standard modules in category $\mathcal O$ for the associated rational Cherednik algebra. In the generic case these coincide with the simple extensions, and their supports are independent of the choice of equivariant quantization.

\end{exmp} 

\section{Koszul Duality and local generators}

We begin by recalling that $\Good{\cW}$ is the abelian category of $\Gm$-equivariant good $\cW$-modules and we set $\QCoh{\cW} = \Ind \Good{\cW}$.  Here, and for the remainder of the article, $\Gm$ denotes the elliptic action. We will often take $\Gm$-invariants of an $\Gm$-equivariant sheaf $\ms{F}$. Let $\resol_{\Eq}$ denote the space $\resol$, equipped with the \textit{equivariant} topology, where the open sets are the $\Gm$-stable open sets in $\resol$. Then $\ms{F}^{\Gm}$ is the sheaf on $\resol_{\Eq}$ defined by $\ms{F}^{\Gm}(U) = (\beta_{\idot} \ms{F})(U)^{\Gm}$. Here $\beta \colon \resol \to \resol_{\Eq}$ is the map of topological spaces which is the identity on points.

\subsection{$\Omega$-modules}\label{sec:Omegamodules}

Let $\mc{E}$ be a bounded complex of (finite rank) locally free good right $\cW$-modules. Let $\mathcal E^\vee = \underline{\sHom}_{\mathcal W_\resol}(\mathcal E,\mathcal W_\resol)$, and $\Omega_{\hbar} = \underline{\sHom}_{\mathcal W_\resol}(\mathcal E,\mathcal E)$. Note that if $\mc{E}$ is a locally free resolution of an object $\mc{L} \in D^b(\Good{\cW})$ then 
\begin{equation}\label{eq:cohomologyOmega}
    \mc{H}^{\idot}(\Omega_{\hbar}) \cong \mc{E}xt_{\cW}^{\idot}(\mc{L},\mc{L}).
\end{equation}
The sheaf $\Omega_{\hbar}$ is a sheaf of dg-rings equipped with a topological $\mathbb{G}_m$-action, which possesses an invertible element of degree $1$ (namely $\hbar$ itself). We define $\Omega = \Omega_{\hbar}^{\mathbb{G}_m}$, which is a $\mathbb{C}$-linear sheaf of dg-rings on $\resol_{\Eq}$. If $\ms{N}$ is a dg $\Omega$-module then 
\[
\widetilde{\ms{N}} := \Omega_{\hbar} \o_{\beta^{-1} \Omega} \beta^{-1} \ms{N}
\]
is a $\Gm$-equivariant dg $\Omega_{\hbar}$-module. This establishes an equivalence $\Lmod{\Omega} \iso \Lmod{(\Omega_{\hbar},\Gm)}$, given by $\ms{N} \mapsto \widetilde{\ms{N}}$, between the category of coherent dg $\Omega$-modules and coherent $\Gm$-equivariant dg $\Omega_{\hbar}$-modules, with inverse $\ms{F} \mapsto \ms{F}^{\Gm}$.  

\begin{lemma}\label{lem:EoverOmega}
Let $\mathcal L$ and $\mathcal E$ be as above, with $\mc{L}$ supported on $\Lambda$. Then 

\begin{enumerate}
\item $\Omega_{\hbar} \cong \mathcal E\otimes_{\cW} \mathcal E^\vee$, as sheaves of dg-algebras.
\item Suppose that the object $\mc{L}$ is nonzero. Then $\cW \cong \mc{E}^{\vee} \o_{\Omega_{\hbar}}\mc{E}$ as dg-algebras, where the differential on $\cW$ is trivial. 
\item $\Omega$ and $\Omega_{\hbar}$ are cohomologically supported on $\Lambda$: that is, the total cohomology $\mathcal H^{\idot}(\Omega)$ is supported on $\Lambda$, and the same is true for $\Omega_{\hbar}$. 
\item $\mathcal E$ is a finite flat left dg $\Omega_{\hbar}$-module.
\end{enumerate}
\end{lemma}

\begin{proof}
Part (1) is immediate. 

Part (2). The evaluation map is a morphism of dg-algebras $\mc{E}^{\vee} \o_{\Omega_{\hbar}}\mc{E} \to \cW$. To show that it is an isomorphism, it suffices to check locally that this is an isomorphism of sheaves (forgetting the grading and differential). In this case we may assume that $\mc{E} = \cW^{\oplus k}$ for some $k>0$ and the claim is clear. 

Part (3) follows from \eqref{eq:cohomologyOmega} since $\mc{L}$ assumed to be supported on $\Lambda$. 

Part (4) is a local statement, so we may assume that $\mc{E}$ is a finite complex of free $\cW$-modules of finite rank. Moreover, we may forget the differential and the (homological) grading when checking that $\mc{E}$ is finite and flat. Thus, we are reduced to the fact that $\mc{E} = \cW^{\oplus k}$ is a finite flat (in fact projective) module over $\sEnd_{\cW}(\cW^{\oplus k})$.  
\end{proof}

\begin{remark}
Let $\mc{E}$ be a bounded complex of locally free good right $\cW$-modules.
\begin{enumerate}
    \item It is essential in Lemma~\ref{lem:EoverOmega} to work sheaf-theoretically. If $P$ is a finitely generated projective module over a (say regular, Noetherian) ring then it is not usually flat over $\End_R(P)$. 
    \item $\mathcal E$ will \textit{not} in general be $K$-flat (in the sense of Spaltenstein) as a left dg $\Omega_{\hbar}$-module. 
\end{enumerate}
\end{remark}

Let $\Lmod{\Omega}$ denote the category of sheaves of coherent right dg $\Omega$-modules and $\LMod{\Omega} = \Ind \Lmod{\Omega}$. For $\ms{M} \in C^b(\Good{\cW})$ and $\ms{N} \in \Lmod{\Omega}$, we define 
\[
G(\ms{M}) = \underline{\sHom}_{\mathcal W_\resol}(\mc{E},\ms{M})^{\Gm}, \quad F(\ms{N}) = \widetilde{\ms{N}} \o_{\Omega_{\hbar}} \mc{E}.
\]

\begin{lem}\label{lem:homotopyequiv}
    The functors $G,F$ define equivalences $C^b(\Good{\cW}) \iso \Lmod{\Omega}$ and $\Lmod{\Omega} \iso C^b(\Good{\cW})$. These induce equivalences 
    \[
    K^b(\Good{\cW}) \iso K(\Lmod{\Omega}), \quad K(\Lmod{\Omega}) \iso K^b(\Good{\cW})
    \]
    and
     \[
     K(\QCoh{\cW}) \iso K(\LMod{\Omega}), \quad K(\LMod{\Omega}) \iso K(\QCoh{\cW}). 
    \]
\end{lem}

\begin{proof}
Let $\ms{N} \in \Lmod{\Omega}$. Then the adjunction $\Id \to GF$ applied to $\ms{N}$ is
\begin{align*} 
\ms{N} \to GF(\ms{N}) & = \sHom_{\cW}(\mc{E}, \widetilde{\ms{N}} \o_{\Omega_{\hbar}} \mc{E})^{\Gm} \\
& \cong (\widetilde{\ms{N}} \o_{\Omega_{\hbar}} \sHom_{\cW}(\mc{E}, \mc{E}) )^{\Gm} \\
& \cong (\widetilde{\ms{N}} \o_{\Omega_{\hbar}} \Omega_{\hbar} )^{\Gm} \cong \ms{N},
\end{align*}
where we have used Lemma~\ref{lem:EoverOmega}(1). Similarly, if $\ms{M} \in C^b(\Good{\cW})$ then 
\begin{align*} 
\ms{M} \leftarrow FG(\ms{M}) & = \widetilde{\sHom_{\cW}(\mc{E},\ms{M})^{\Gm}} \o_{\Omega_{\hbar}} \mc{E} \\
& \cong \sHom_{\cW}(\mc{E},\ms{M}) \o_{\Omega_{\hbar}} \mc{E} \\
& \cong \ms{M} \o_{\cW} (\sHom_{\cW}(\mc{E},\cW) \o_{\Omega_{\hbar}} \mc{E}) \\
& \cong \ms{M} \o_{\cW} \cW = \ms{M},
\end{align*}
where we have used the fact that if $\ms{P}$ is a $\Gm$-equivariant dg $\Omega_{\hbar}$-module then the adjunction $\widetilde{\ms{P}^{\Gm}} \to \ms{P}$ is an isomorphism and we have used Lemma~\ref{lem:EoverOmega}(2). Since $F,G$ are additive functors they induce equivalences of homotopy categories. 

Identifying $C(\QCoh{\cW}) = \Ind C^b(\Good{\cW})$ and using the fact that $G$ and $F$ are cocontinuous on $C^b(\Good{\cW})$ and $\Lmod{\Omega}$ respectively, the above proof shows that $G$ and $F$ extend to equivalences $C(\QCoh{\cW}) \iso \LMod{\Omega}$ and $\LMod{\Omega} \iso C(\QCoh{\cW})$ respectively. Passing to homotopy categories gives the final claim of the lemma. 
\end{proof}

If $\ms{F}$ is a complex of sheaves on $\resol$ then we say that $\ms{F}$ is \textit{acyclic} if the sheaves $\mc{H}^i(\ms{F})$ are zero for all $i \in \Z$. Since $\mc{E}$ is a bounded complex of locally free $\cW$-modules, the functor $G$ is $K$-projective i.e. if $\ms{M} \in K(\QCoh{\cW})$ is an acyclic complex then so too is $G(\ms{M})$. However, $\mc{E}$ will almost never be $K$-flat as an $\Omega$-module and so, even if $\ms{N}$ is acyclic, $F(\ms{N})$ will not in general be acyclic. An object $\ms{N} \in K(\LMod{\Omega})$ is said to be \textit{$F$-acyclic} if $F(\ms{N})$ is acyclic. 

\begin{lem}\label{lem:Facyclicacyclic}
If $\ms{N} \in K(\LMod{\Omega}) $ is $F$-acyclic then $\ms{N}$ is acyclic.
\end{lem}

\begin{proof}
This follows from the fact that the adjunction $\ms{N} \to GF(\ms{N})$ is an isomorphism in $K(\LMod{\Omega})$ and $G$ sends acyclic objects to acyclic objects. 
\end{proof}

Let $\mc{N}_F$ denote the full subcategory of $K(\LMod{\Omega})$ consisting of $F$-acyclic complexes. Then $\mc{N}_F$ is a thick strictly triangulated subcategory of $K(\LMod{\Omega})$ and we call the Verdier quotient $K(\LMod{\Omega}) / \mc{N}_F$ the "exotic derived category" $D_{ex}(\LMod{\Omega})$.

Since $G,F$ are equivalences on the level of homotopy categories, they induce equivalences 
\begin{equation}\label{eq:DOmegaequi}
    G \colon D(\QCoh{\cW}) \iso D_{ex}(\LMod{\Omega}), \quad F \colon  D_{ex}(\LMod{\Omega}) \iso D(\QCoh{\cW}).
\end{equation}
Let $D^b_{ex}(\Lmod{\Omega})$ denote the quotient of $K(\Lmod{\Omega})$ by $\mc{N}_F \cap K(\Lmod{\Omega})$. Then $G,F$ restrict to equivalences
\[
G \colon D^b(\Good{\cW}) \iso D_{ex}^b(\Lmod{\Omega}), \quad F \colon  D_{ex}^b(\Lmod{\Omega}) \iso D^b(\Good{\cW}). 
\]  
Note that Lemma~\ref{lem:Facyclicacyclic} implies that there are quotient functors $D_{ex}(\LMod{\Omega}) \to D(\LMod{\Omega})$ and $D_{ex}^b(\Lmod{\Omega}) \to D^b(\Lmod{\Omega})$. 

\begin{remark}
    This construction can be very degenerate. For instance, consider the two-term complex $\mc{E} \colon \cW \stackrel{\Id}{\longrightarrow} \cW$. Then every $\Omega$-module is acyclic and $G$ is the (essentially surjective) functor that sends an object $\ms{M}$ to the $\Gm$-invariants of the cone $\mathrm{Cone}(\Id)$ of the identity morphism $\ms{M} \to \ms{M}$. Then one can easily check that $\ms{N} \cong G(\ms{M})$ is $F$-acyclic if and only if $\ms{M}$ is acyclic.   
\end{remark}

\begin{lem}
 $D^b_{ex}(\Lmod{\Omega})$ is a full subcategory of $D_{ex}(\LMod{\Omega})$. 
 \end{lem}
 
 \begin{proof}
 This follows from the fact that we can form the commutative diagram 
 \[
 \begin{tikzcd}
 D^b_{ex}(\Lmod{\Omega}) \ar[r] \ar[d,"F"'] & D_{ex}(\LMod{\Omega}) \\
 D^b(\Good{\cW}) \ar[r] & D(\QCoh{\cW}) \ar[u,"G"']
 \end{tikzcd}
 \]
 with the vertical arrows being equivalences and the bottom horizontal arrow fully faithful.
 \end{proof}

\begin{defn}\label{defn:localgenerator}
    The object $\mc{L} \in D^b(\Good{\cW})$ is a \textit{local generator} if given $\ms{M} \in D^b(\Good{\cW})$ such that $\R Hom_{\cW}(\mc{L},\ms{M})$ is an acyclic sheaf, we must have $\ms{M} \cong 0$.
\end{defn}

\begin{lem}\label{lem:acyclicpointsRHom}
Let $\mc{L}, \ms{M} \in D^b(\Good{\cW})$. Under the elliptic $\Gm$-action on $\resol$, the complex of sheaves $\R \sHom_{\cW}(\mc{L},\ms{M})$ is acyclic if and only if $\R \sHom_{\cW}(\mc{L},\ms{M})^{\Gm}$ is acyclic on $\resol_{\Eq}$ if and only if $\R \sHom_{\cW}(\mc{L},\ms{M})^{\Gm}_y$ is an acyclic complex of vector spaces for all $y \in Y$. 
\end{lem}

\begin{proof}
Note that $\R \sHom_{\cW}(\mc{L},\ms{M})$ is acyclic if and only if the stalks $\R \sHom_{\cW}(\mc{L},\ms{M})_x$ are acyclic for all $x \in \resol$. The locus in $\resol$ where $\R \sHom_{\cW}(\mc{L},\ms{M})$ is not acyclic is closed and $\Gm$-stable. Since the $\Gm$ action on $\resol$ is elliptic, any such closed subset contains a fixed point. Therefore, if $\R \sHom_{\cW}(\mc{L},\ms{M})$ is not acyclic then there exits $y \in Y$ such that $\R \sHom_{\cW}(\mc{L},\ms{M})_y$ is not acyclic. Now, $\R \sHom_{\cW}(\mc{L},\ms{M})_y$ is a complex of $\Gm$-equivariant $\C(\!(\hbar)\!)$-modules with $\C(\!(\hbar)\!)$-linear differential. Since $\Gm$ acts with weight one on $\hbar$, the complex $\R \sHom_{\cW}(\mc{L},\ms{M})_y$ is acyclic if and only if $\R \sHom_{\cW}(\mc{L},\ms{M})_y^{\Gm}$ is acyclic. Thus, if $\R \sHom_{\cW}(\mc{L},\ms{M})$ is not acyclic then there exists $y \in Y$ such that $\R \sHom_{\cW}(\mc{L},\ms{M})_y^{\Gm}$ is not acyclic. Conversely, if the complex $\R \sHom_{\cW}(\mc{L},\ms{M})$ is acyclic then clearly $\R \sHom_{\cW}(\mc{L},\ms{M})_y^{\Gm}$ is acyclic for all $y \in Y$. 

A similar argument shows that $\R \sHom_{\cW}(\mc{L},\ms{M})^{\Gm}$ is acyclic if and only if $\R \sHom_{\cW}(\mc{L},\ms{M})^{\Gm}_y$ is an acyclic complex of vector spaces for all $y \in Y$. 
\end{proof}

\begin{lem}\label{lem:localgeneratorderivedex}
	Let $\mc{L} \in D^b(\Good{\cW})$ be a local generator. Then $D_{ex}^b(\Lmod{\Omega}) \iso D^b(\Lmod{\Omega})$. 
\end{lem}

\begin{proof}
We must show that $\ms{N} \in K(\Lmod{\Omega})$ is $F$-acyclic if and only if $\ms{N}$ is acyclic. One implication always holds by Lemma~\ref{lem:Facyclicacyclic}. Therefore, we need to show that if $\ms{N}$ is acyclic then it is $F$-acyclic. By Lemma~\ref{lem:homotopyequiv}, we may assume that $\ms{N} = G(\ms{M})$ for some $\ms{M} \in K^b(\Good{\cW})$. This means that $F(\ms{N}) \cong \ms{M}$. Moreover, 
	\[
	\mc{H}^i(\ms{N}) = \mc{H}^i(\sHom_{\cW}(\mc{E},\ms{M})^{\Gm}) = \R \sHom^i_{\cW}(\ms{L},\ms{M})^{\Gm}. 
	\]	
	Therefore, if $\mc{H}^i(\ms{N}) = 0$ for all $i$ then $\R \sHom_{\cW}(\ms{L},\ms{M})^{\Gm}$ is an acyclic complex of sheaves. By Lemma~\ref{lem:acyclicpointsRHom}, $\R \sHom_{\cW}(\ms{L},\ms{M})$ is acyclic and we deduce that $\ms{M}$ is zero since $\ms{L}$ is a local generator. 
\end{proof}

Note that when $\mc{E}$ is a resolution of a local generator $\mc{L}$, Lemma~\ref{lem:localgeneratorderivedex} shows that $\mc{E}$ is $K$-flat as an object in $K(\Lmod{\Omega})$, even though it is not $K$-flat in $K(\LMod{\Omega})$. 

We have shown in the proof of Lemma~\ref{lem:localgeneratorderivedex} that $\mc{N}_F \cap K(\Lmod{\Omega})$ is the category $\mathrm{Acyc}(\Lmod{\Omega})$ of acyclic objects in $K(\Lmod{\Omega})$. The following shows that one can characterize $\mc{N}_F$ entirely in terms of these acyclic complexes. 

\begin{lem}\label{lem:colimitofacyclic}
	Every acyclic complex $\ms{M} \in K(\QCoh{\cW})$ is the filtered colimit of a sequence of coherent (bounded) acyclic complexes $\ms{M}_i \in K^b(\Good{\cW})$. 
	
	Conversely, if $\ms{M} \in K(\QCoh{\cW})$  is the filtered colimit of a sequence of coherent (bounded) acyclic complexes $\ms{M}_i \in K^b(\Good{\cW})$ then $\ms{M}$ is acyclic. 
\end{lem}

\begin{proof}
	The first claim follows from the fact that $\cW$ is a coherent sheaf of rings. 
	
	The converse follows from \cite[Lemma~10.8.8]{stacks}. 
\end{proof}

\begin{prop}
	The category $\mc{N}_F$ is the smallest thick strictly triangulated subcategory of the homotopy category $K(\LMod{\Omega})$ containing $\mathrm{Acyc}(\Lmod{\Omega})$ and closed under coproducts. 
\end{prop}

\begin{proof}
	Let $\ms{N}$ be $F$-acyclic. Then $\ms{M} = F(\ms{N})$ is an acyclic complex of $\cW$-modules. By Lemma~\ref{lem:colimitofacyclic}, we may write $\ms{M} = \lim_{\to} \ms{M}_i$ as a filtered colimit of acyclic objects in $K^b(\Good{\cW})$. Since $G ( - ) = \sHom_{\cW}(\mc{E},-)$ and $\mc{E}$ is a bounded complex of (coherent) locally free $\cW$-modules, $G$ preserves filtered colimits (see e.g. \cite[Theorem~4.5]{CHDirect}; in fact, by \cite[Lemma~3.4]{CHDirect}, the functor $G$ preserved arbitrary colimits). Therefore, $\ms{N} = \lim_{\to} \ms{N}_i$, where $\ms{N}_i = G(\ms{M}_i)$ is acyclic in $K(\Lmod{\Omega})$. Therefore, each $\ms{N}_i$ is in $\mathrm{Acyc}(\Lmod{\Omega})$. 
	
	Conversely, if $\ms{N} = \lim_i \ms{N}_i$ with each $\ms{N}_i$ coherent and acyclic then it is a filtered colimit of $F$-acyclic complexes. Then $F(\ms{N}) = \lim_{\to} F(\ms{N}_i)$ because $F$ is is a left adjoint and hence commutes with arbitrary colimits. Thus, $F(\ms{N})$ is a colimit of acyclic complexes in $K(\QCoh{\cW})$. This implies that $F(\ms{N})$ is acyclic by \cite[Lemma~10.8.8]{stacks}. 	
\end{proof}

\subsection{Existence of a local generator} 

Fix a total ordering on $I$ such that $K_{< r} := \bigsqcup_{i < r} C_i$ is closed in $\resol$. Let $U_r$ denote the complement and $j_r \colon U_r \to \resol$ the open embedding. We note that $U_r$ is generally strictly bigger that the set $U(r)$ of Section~\ref{sec:externsionsympresmod}, and the former is not usually affine. We set:
\[
\mc{L} = \bigoplus_{r \in I} \R (j_r)_* \delta_{\Lag_r}.
\]
By Theorem~\ref{thm:holonomicpushforward}, $\mc{L}$ is a bounded complex with holonomic cohomology. 

Our goal is to prove the following theorem.

\begin{thm}\label{thm:localgeneratorv3}
	The object $\mc{L}$ is a local generator. 
\end{thm}

The proof is by induction on the number of coisotropic cells in $\resol$. We begin by assuming that $\resol$ itself is a coisotropic cell. This means that $\resol = T^* V$ for some vector space $V$. In this case the elliptic $\Gm$-action scales the fibers of $T^*V$ and hence has fixed points $Y = V$. Let $\rho \colon T^* V \to V$ be the projection map. By \cite[Proposition~4.33]{BDMN}, the functor $\ms{M} \mapsto (\rho_{\idot} \ms{M})^{\Gm}$ is an equivalence $\Good{\cW} \iso \Lmod{\dd_V}$, the latter being the category of coherent right $\dd_V$-modules. 

\begin{lem}\label{lem:invarianthom}
Assume that $\resol = T^* V$. Let $x \in V \subset T^* V$ and $\ms{M}_i \in D^b(\Good{\cW_{\resol,x}})$ for $i = 1,2$. Then $\cW_{\resol,x}^{\Gm} = \dd_{V,x}$ and
\begin{equation}\label{eq:invariantsiso}
 \R \Hom_{\cW_{\resol,x}}(\ms{M}_1, \ms{M}_2)^{\Gm} \cong \R \Hom_{\dd_{V,x}}(\ms{M}_1^{\Gm}, \ms{M}_2^{\Gm}). 
 \end{equation}
 \end{lem}
 
 \begin{proof}
 In this case, $\cW_\resol$ can be identified as a sheaf of $\Gm$-equivariant $\C(\!(\hbar)\!)$-modules with $\mc{O}_\resol(\!(\hbar)\!)$, where the product is given by the Moyal-Weyl product on $\mc{O}_\resol(\!(\hbar)\!)$. Any section $s \in \cW_{\resol,x}^{\Gm}$ is defined on some $\Gm$-invariant open neighbourhood of $x$ and shrinking this if necessary, we may assume $s \in \mc{O}_\resol(U \times V^*) \widehat{\o} \C(\!(\hbar)\!)$, where $x \in U \subset V$. It is easily checked that any such invariant section is of the form
 \[
 s = \hbar^{-N} D_{-N}(x,y) + \hbar^{-N+1} D_{-N+1}(x,y) + \cdots + D_0(x,y),
 \]
 where each $D_{-i}(x,y) = \sum_{|\alpha| = i} f_i(x) y^{\alpha}$ is a polynomial homogeneous in the $y$'s of degree $i$ and $f_i(x) \in \mc{O}_V(U)$. Since the map $\dd_{V,x} \to  \cW_{\resol,x}^{\Gm}$ sends $\partial_i$ to $\hbar^{-a_i} y_i$, it follows that this is an isomorphism. 
 
 Since $\ms{M}_1 \in D^b(\Good{\cW_{\resol,x}})$, it is isomorphic to a finite complex of finite rank projective $\cW_{\resol,x}$-modules, where the differentials are $\Gm$-equivariant. Then the identification \eqref{eq:invariantsiso} reduces to the case where $\ms{M}_1 = \cW_{X,x}\{ j \}$ for some $j \in \Z$. 
 \end{proof}
 
We continue to assume that $\resol = T^* V$. Let $n = \dim V$ and let $\Omega_{V,\hbar}^n$ denote the good right $\cW$-module corresponding to $\Omega^n_V$ under the equivalence $\Good{\cW} \iso \Lmod{\dd_V}$. Note that $\Omega_{V,\hbar}^n$ is a simple $\mathsf{T}$-equivariant holonomic object supported on $V$. Therefore, Proposition~\ref{prop:uniquesimple} says that, up to twisting by a linear character of $\mathsf{T}$, $\delta_V := \Omega_{V,\hbar}^n$ is the unique such simple object. 

\begin{prop}\label{prop:localgeneratorcotangent}
	The sheaf $\delta_V$ is a local generator in $D^b(\Good{\cW})$. 
\end{prop}

\begin{proof}
Let $\ms{M} \in D^b(\Good{\cW})$. By definition, we must show that if $\R \sHom_{\cW}(\delta_V, \ms{M})$ is acyclic then $\ms{M} \cong 0$. The complex of sheaves is acyclic if and only if the complex of vector spaces $\R \sHom_{\cW}(\delta_V, \ms{M})_x$ is acyclic for all $x \in \resol$. Since $\cW$ is a coherent sheaf of algebras and $\delta_V$ a coherent $\cW$-module, \cite[Proposition~A.3]{Kashiwara} says that 
\begin{equation}\label{eq:RHomOmegaV}
\R \sHom_{\cW}(\delta_{V}, \ms{M})_x = \R \Hom_{\cW_x}(\delta_{V,x}, \ms{M}_x).
\end{equation}
Take $x \in V \subset T^* V = \resol$. Then $\Gm$ acts on the spaces of \eqref{eq:RHomOmegaV}. Taking $\Gm$-invariants of the right hand side gives   
\begin{align*}
 \R \Hom_{\cW_x}(\delta_{V,x}, \ms{M}_x)^{\Gm} & = \R \Hom_{\cW_x}((\Omega_{V,\hbar}^n)_x, \ms{M}_x)^{\Gm} \\
 & =  \R \Hom_{\cW_x^{\Gm}}((\Omega_{V,\hbar}^n)_x^{\Gm}, \ms{M}_x^{\Gm}) \\
 & \cong \R \Hom_{\dd_{V,x}}(\Omega_{V,x}^n, \ms{M}_x^{\Gm}) \\
 & = \R \sHom_{\dd_{V}}(\Omega_{V}^n, \ms{M}^{\Gm})_x
 \end{align*}
  having applied Lemma~\ref{lem:invarianthom}. This implies that 
  \[
\R \Hom_{\dd_{V}}(\Omega_{V}^n, \ms{M}^{\Gm}) \cong \ms{M}^{\Gm} \o_{\dd_V}^{\mathbb{L}} \mc{O}_V
\]
is an acyclic complex of sheaves on $V$. Hence, by \cite[Lemma~7.2.6(iii)]{BD} or \cite[Lemma~2.1.6]{BDChiral}, $\ms{M}^{\Gm} \cong 0$. But taking $\Gm$-invariants is an equivalence, hence $\ms{M} \cong 0$. 
\end{proof}

Now, we prove by induction on the strata that $\mc{L}$ is a local generator. Let $C \subset \resol$ be a closed coisotropic cell and $U$ its complement. Assume that $\ms{M} \in D^b(\Good{\cW})$ such that the complex $\R \sHom_{\cW}(\mc{L},\ms{M})$ is acyclic. Then 
\[
\R \sHom_{\cW}(\mc{L},\ms{M}) |_U = \R \sHom_{\cW_U}(\mc{L} |_U,\ms{M} |_U)
\]
is acyclic and hence we may assume by induction that $\ms{M} |_U \cong 0$. Therefore, $\ms{M}$ is supported on $C$ and Kashiwara's Theorem \cite[Theorem~4.28]{BDMN} says that there exists $\ms{M}'$ on $S$ such that $\ms{M} = \Ham^{\perp}(\ms{M}')$. Then 
\[
\R \sHom_{\cW}(\mc{L},\ms{M}) = \R \sHom_{\cW}(\mc{L},\Ham^{\perp} (\ms{M}') ).
\]
Pick a point $x \in C$. Then $\R \sHom_{\cW}(\mc{L},\Ham^{\perp} (\ms{M}') )_x$ is an acyclic sheaf of vector spaces. 

\begin{lem}\label{lem:stalksofRHom}
Let $y \in Y \subset C$ and $s = \pi(y) \in S$. Then 
\[
\R \sHom_{\cW}(\mc{L},\Ham^{\perp} (\ms{M}') )_y \cong \R \sHom_{\cW_S}(\Ham^*(\mc{L}),\ms{M}')_s.
\]
\end{lem}

\begin{proof}
	Let us say that an open subset $U \subset \resol$ is $C$-good if (a) $U$ is a $\Gm$-stable affine open set and (b) $C \cap U_i = \{ u \in U_i \, | \, \lim_{t \to \infty} t \cdot u \in Y \}$. By \cite[Lemma~4.25]{BDMN}, $\resol$ admits a cover by $C$-good open sets. Thus, we may assume that $\resol$ is itself $C$-good. 
	
	Moreover, we can fix a filtered system $\{ U_i \}_{i \in I}$ of $C$-good neighbourhoods of $y$, meaning that $U_j \subset U_i$ if $i \to j$ and $\bigcap_{i \in I} U_i = \{ y \}$. Indeed, Sumihiro's Theorem implies that we can choose a filtered system $\{ U_i \}_{i \in I}$ of affine $\Gm$-stable open neighbourhoods of $y$. Then it suffices to note that if $C = \{ x \in \resol \, | \, \lim_{t \to \infty} t \cdot x \in Y \}$ then $C \cap U = \{ u \in U \, | \, \lim_{t \to \infty} t \cdot u \in Y \}$ for all $\Gm$-stable open subsets of $\resol$. 
	
	Next, we claim that $\{ S_i := \pi(C \cap U_i) \}_{i \in I}$ is a filtered system of $\Gm$-stable affine open neighbourhoods of $s$ in $S$. It follows from the explicit construction of $\pi$ given in Section 2.4 of \cite{BDMN} that $S_i = \mathrm{Spec} ( (R/I)^{ \{ I , - \} })$, where $R = \mc{O}(U_i)$ and $I$ is the radical ideal in $R$ defining $C$. Thus, each $S_i$ is an affine $\Gm$-stable open open neighbourhood of $s$. If $s \in U \subset S$ then there exists an open set $U' \subset \resol$ containing $y$ such that $U' \cap C = \pi^{-1}(U)$. If $i \in I$ such that $U_i \subset U'$ then $S_i \subset U$. 
	
	Finally, we compute:
	\begin{align*}
\R \sHom_{\cW}(\mc{L},\Ham^{\perp} (\ms{M}') )_y & = \lim_I \R \sHom_{\cW_{X}}(\mc{L},\Ham^{\perp} (\ms{M}'))(U_i) \\
& \cong \lim_I \R \Hom_{\cW_{X}(U_i)}(\mc{L}(U_i),\Ham^{\perp} (\ms{M}')(U_i)) \\
& \cong \lim_I \R \Hom_{\cW_{X}(U_i)}(\mc{L}(U_i),\Ham^{\perp} (\ms{M}'(S_i))) \\
& \cong \lim_I \R \Hom_{\cW_{S}(S_i)}(\Ham^*(\mc{L})(S_i),\ms{M}'(S_i)) \\
& \cong \lim_I \R \sHom_{\cW_{S}}(\Ham^*(\mc{L}),\ms{M}')(S_i) \\
& = \R \sHom_{\cW_S}(\Ham^*(\mc{L}),\ms{M}')_s
\end{align*}	 
where we have applied the adjunction of Proposition~\ref{prop:adjunctionforHamperp} in the fourth line.

\end{proof}

\begin{prop}\label{prop:RHomacycliccoistred}
	If $\R \sHom_{\cW}(\mc{L},\Ham^{\perp} (\ms{M}') )$ is acyclic then so too is $\R \sHom_{\cW_S}(\Ham^*(\mc{L}),\ms{M}')$.
\end{prop}

\begin{proof}
Note that the action of $\Gm$ on $S$ is elliptic. Therefore, Lemma~\ref{lem:localgeneratorderivedex} says that the complex $\R \sHom_{\cW_S}(\Ham^*(\mc{L}),\ms{M}')$ is acyclic if and only if $\R \sHom_{\cW_S}(\Ham^*(\mc{L}),\ms{M}')_s$ is an acyclic complex of vector spaces for all $s \in S^{\Gm}$. Fix such a point $s$ and let $y \in Y \subset C$ such that $\pi(y) = s$. Then Lemma~\ref{lem:stalksofRHom} says that 
 \[
 \R \sHom_{\cW}(\mc{L},\Ham^{\perp} (\ms{M}') )_y \cong \R \sHom_{\cW_S}(\Ham^*(\mc{L}),\ms{M}')_s.
 \]
 Since we have assumed that $\R \sHom_{\cW}(\mc{L},\Ham^{\perp} (\ms{M}') )_y$ is acyclic, the proposition follows. 
\end{proof}





We deduce that $\R \sHom_{\cW_S}(\Ham^*(\mc{L}),\ms{M}')$ is an acyclic complex of sheaves. 

\begin{lem}\label{lem:pullbacklocalgenerator}
    $\Ham^*(\mc{L})$ is a local generator on $S$. 
\end{lem}

\begin{proof}
Let $N = |I|$. If, for $r < N$, $j_r' \colon U_r \hookrightarrow U$ then $j_r = j \circ j_r'$. Then 
\[
\mc{L} = \delta_{\mathsf{\Lambda}_{|I|}} \oplus \R j_* \left( \bigoplus_{r < N} \R (j_r')_* \delta_{\mathsf{\Lambda}_{r}} \right).
\]
Since $\R i^! \circ \R j_* = 0$, we deduce that $\R i^! (\mc{L}) = \delta_{\mathsf{\Lambda}_{N}}$ and hence $\Ham^* \mc{L}= \delta_{\mathsf{\Lambda}_N}$ on $S = S_N$. Then the claim follows from Proposition~\ref{prop:localgeneratorcotangent} applied to $S \cong T^* V_N$. 
\end{proof}

We deduce from Proposition~\ref{prop:RHomacycliccoistred} and Lemma~\ref{lem:pullbacklocalgenerator} that $\ms{M}' = 0$. This completes the proof of Theorem~\ref{thm:localgeneratorv3}. 

\begin{remark}
    Since $\mc{L} \in D^b_{\hol}(\Good{\cW})$, equation~\eqref{eq:cohomologyOmega} implies that $\Omega$ is a proper dg algebra. 
\end{remark}

\subsection{The coderived category}

This subsection is an aside, and not used elsewhere in the paper. 

Implicitly, the proof of Proposition~\ref{prop:localgeneratorcotangent} is based on classical $\dd$-$\Omega$-duality, which is an equivalence $D^b(\Lmod{\dd_V}) \iso D_{ex}(\Lmod{\Omega_{\DR}})$, originally considered by Kapranov \cite{Kap} and later Saito \cite{Saito} and Beilinson-Drinfeld \cite{BD,BDChiral}. Here $\Omega_{\DR}$ is the (sheaf) de Rham dg-algebra of poly-differential forms on $V$. More recently, Positselski \cite[Appendix B]{Po} has given a beautiful reformulation (and generalization) of this result in terms of coderived categories. Namely, one first considers the smallest thick triangulated subcategory $\text{Acyc}^\text{co}(\Omega_{\DR})$ of $K(\LMod{\Omega_{\DR}})$ that is cocomplete and contains the total complex of every exact triple in $C(\Omega_{\DR})$. Then Positselski defines the coderived category $D^{\co}(\Omega_{\DR})$ as the Verdier quotient of $K(\LMod{\Omega_{\DR}})$ by $\text{Acyc}^\text{co}(\Omega_{\DR})$. He proves that the $(\Omega_{\DR},\dd_V)$-bimodule $\mc{T} := \Omega_{\DR} \o_{\mc{O}} \dd_V$ (equipped with the de Rham differential) induces an equivalence 
\[
    G_{\DR} \colon D(\LMod{\dd_V}) \iso D^{\co}(\Omega_{\DR}), \quad G_{\DR}(\ms{M}) = \underline{\sHom}_{\dd_V}(\mc{T},\ms{M})
\]
with quasi-inverse $F_{\DR}(\ms{N}) = \ms{N} \o_{\Omega_{\DR}} \mc{T}$. 

If we define $\Omega = \underline{\sEnd}_{\dd_V}(\mc{T})$ then there is a quasi-isomorphism $\Omega_{\DR} \iso \Omega$ of sheaves of dg-algebras on $V$. If we define the thick subcategories $\mc{N}_F \subset K(\LMod{\Omega})$ and $\mc{N}_{F_{\DR}} \subset K(\LMod{\Omega_{\DR}})$ as in Section~\ref{sec:Omegamodules}, then Positselski's result implies that $\mc{N}_{F_{\DR}} = \text{Acyc}^\text{co}(\Omega_{\DR})$. Moreover, one can show that:

\begin{prop}\label{prop:Poisequivfactor}
    Positselski's equivalences $G_{\DR},F_{\DR}$ can be factored as 
\[
\begin{tikzcd}
D(\dd_V) \ar[r,bend right,"G"'] & K(\LMod{\Omega}) / \mc{N}_{\Omega} \ar[r,bend right,"\Res"'] \ar[l,bend right,"F"'] & D^{\co}(\Omega_{\DR}) \ar[l,bend right,"\Ind"'],  
\end{tikzcd}
\]
where all functors are equivalences. 
\end{prop}

\noindent We omit the proof since the result is not needed elsewhere. 

Based on Proposition~\ref{prop:Poisequivfactor} one would expect that $K(\LMod{\Omega}) / \mc{N}_{\Omega} = D^{\co}(\LMod{\Omega})$. However, just as in Lemma~\ref{lem:EoverOmega}, one can check that the sheaf $\Omega$ has finite global dimension since $\mc{T}$ is a finite free complex of $\dd_V$-modules. Therefore, it follows from \cite[Theorem~3.6(a)]{Po} that $D^{\co}(\LMod{\Omega}) = D(\LMod{\Omega})$ is the usual derived category. Since $\Omega$ and $\Omega_{\DR}$ are quasi-isomorphic, we would then conclude from Proposition~\ref{prop:Poisequivfactor} that $D^{\co}(\Omega_{\DR}) = D(\Omega_{\DR})$ too. But this cannot be the case, as explained for instance in \cite[Section~7]{BD}. Thus, $K(\LMod{\Omega}) / \mc{N}_{\Omega}$ is not the coderived category. The key point is that quasi-isomorphic dg-algebras need not have equivalent coderived categories.

Globally, if we no longer assume $\resol = T^* V$ then one might similarly assume that the exotic derived category $D_{ex}(\LMod{\Omega})$, appearing in the equivalences of \eqref{eq:DOmegaequi}, is just the coderived category. But, again, Lemma~\ref{lem:EoverOmega} implies that the coderived category of $\Omega$ equals the usual derived category $D(\LMod{\Omega})$ and hence is a proper quotient of $D_{ex}(\LMod{\Omega})$.


\subsection{Symplectic resolutions}

In this section we assume once again that $\mu \colon \resol \to X$ is a symplectic resolution. We first note that locally free resolutions exist. Indeed, by \cite[Corollary~B.1]{BLPWAst} we may assume, up to twisting by a quantized line bundle, that localization holds for our quantization. In this case, every module over the ring of global sections $\Gamma(\resol,\cW)^{\Gm}$ admits a finite resolution by finitely generated projective modules since the ring has finite global dimension. These projective modules localize to finite rank locally free $\cW$-modules. 

By Proposition~\ref{simple extension prop}, for each $1 \le r \le N = |I|$, there exists a simple holonomic $\mathsf{T}$-equivariant (good) $\cW$-module $\mc{L}_r$ that is supported on $\Lag$ and $\mc{L}_r |_{U(r)} \cong \delta_{\Lag_r}$. We take now
\[
\mc{L} := \bigoplus_{r = 1}^N \mc{L}_r. 
\]
If $\mc{E}$ is a locally free resolution of $\mc{L}$ then $\Omega = \underline{\sEnd}_{\cW}(\mc{E})^{\Gm}$ is cohomologically supported on $\Lag$. 



\begin{proof}[Proof of Theorem~\ref{thm:symprescohsupport}]

We just need to show that $\mc{L}$ is a local generator for $D^b(\Good{\cW})$. Assume $\ms{M} \in D^b(\Good{\cW})$ with $\R \sHom_{\cW}(\mc{L},\ms{M})$ acyclic. If $U$ is the open coisotropic cell then $\mc{L} |_U = \mc{L}_1 |_U = \delta_{\Lag_1}$ is a local generator by Proposition~\ref{prop:localgeneratorcotangent}. As in the proof of Theorem~\ref{thm:localgeneratorv3}, we may therefore assume by induction that there is a closed coisotropic cell $C = C_N$ with complement $U$ such that $\ms{M} |_U \cong 0$. Again, $\ms{M} = \Ham^{\perp}(\ms{M}')$ for some good $\cW_S$-module $\ms{M}'$. By Proposition~\ref{prop:RHomacycliccoistred}, we are reduced to showing that $\Ham^*(\mc{L})$ is a local generator on $S$. Note that 
\[
\R \sHom_{\cW_S}(\Ham^*(\mc{L}),\ms{M}') = \bigoplus_{r = 1}^N \R \sHom_{\cW_S}(\Ham^*(\mc{L}_r),\ms{M}')
\]
and if the left hand side is acyclic then every summand $\R \sHom_{\cW_S}(\Ham^*(\mc{L}_r),\ms{M}')$ of the right hand side must be acyclic. In particular, $\R \sHom_{\cW_S}(\Ham^*(\mc{L}_N),\ms{M}')$ is acyclic. But since $C_N$ is closed in $\resol$, $\mc{L}_N = \Ham^{\perp}(\delta_{V_N})$ if we identify $S = S_N = T^* V_N$. Then $\Ham^*(\mc{L}_N) = \Ham^*(\Ham^{\perp}(\delta_{V_N})) \cong \delta_{V_N}$ is a local generator by Proposition~\ref{prop:localgeneratorcotangent}. This forces $\ms{M}' \cong 0$.     
\end{proof}

\bibliographystyle{alpha}

\end{document}